\documentclass{siamart1116}
\usepackage{amssymb}
\usepackage{bm}
\usepackage{datetime}
\usepackage{graphicx}
\usepackage[protrusion=true,tracking=false,kerning=true]{microtype}
\usepackage[figuresonly]{endfloat}
	\renewcommand\figureplace{%
		\begin{center}
		(Insert \figurename~\thepostfigure\ around here. Currently at the end of the manuscript.)
 		\end{center}
	}
\usepackage[algo2e, linesnumbered, ruled, noline, longend]{algorithm2e}

    \SetKwSty{mykwfnt}\DontPrintSemicolon

\newcommand{\TheTitle}{Duality of sum of nonnegative circuit polynomials and optimal SONC bounds}
\newcommand{\TheRunningTitle}{SONC duality and circuit generation}
\newcommand{\TheAuthors}{D{\'a}vid Papp}

\headers{\TheRunningTitle}{\TheAuthors}

\title{{\TheTitle}\thanks{Original manuscript: \today.
\funding{This material is based upon work supported by the National Science Foundation under Grant No.~DMS-1719828 and Grant No.~DMS-1847865.}}}

\author{\TheAuthors\thanks{North Carolina State University, Department of Mathematics. E-mail: \texttt{dpapp@ncsu.edu}. ORCID: 0000-0003-4498-6417}}

\newcommand{\deletethis}[1]{{}}
\newcommand{\bibhack}[1]{}

\numberwithin{theorem}{section}
\theoremstyle{remark}
\newtheorem{example}[theorem]{Example}
\newtheorem{remark}[theorem]{Remark}

\newcommand{\vA}{\mathbf{A}}
\newcommand{\vaf}{{\bm{\alpha}}}
\newcommand{\vbeta}{{\bm{\beta}}}
\newcommand{\vgamma}{{\bm{\gamma}}}
\newcommand{\veta}{{\bm{\eta}}}
\newcommand{\vomega}{{\bm{\omega}}}
\newcommand{\vc}{\mathbf{c}}
\newcommand{\ve}{\mathbf{e}}
\newcommand{\vf}{\mathbf{f}}
\newcommand{\vlam}{{\bm{\lambda}}}
\newcommand{\vone}{\mathbf{1}}
\newcommand{\vv}{\mathbf{v}}
\newcommand{\vw}{\mathbf{w}}
\newcommand{\vx}{\mathbf{x}}
\newcommand{\vxi}{{\bm{\xi}}}
\newcommand{\vy}{\mathbf{y}}
\newcommand{\vz}{\mathbf{z}}
\newcommand{\vzero}{\mathbf{0}}

\newcommand{\CC}{\mathcal{C}}
\newcommand{\defeq}{\ensuremath{\overset{\textup{def}}{=}}}
\newcommand{\N}{\mathbb{N}}
\newcommand{\R}{\mathbb{R}}
\renewcommand{\SS}{\mathcal{S}}
\newcommand{\T}{\mathrm{T}}

\newcommand{\conv}{\operatorname{conv}}
\newcommand{\fai}{\ensuremath{f_{\vaf_i}}}
\newcommand{\New}{\operatorname{New}}
\newcommand{\cP}{\mathcal{P}}
\newcommand{\supp}{\operatorname{supp}}

\begin{document}

\maketitle

\begin{abstract}
Circuit polynomials are polynomials satisfying a number of conditions that make it easy to compute sharp and certifiable global lower bounds for them. Consequently, one may use them to find certifiable lower bounds for any polynomial by writing it as a sum of circuit polynomials with known lower bounds (if possible), in a fashion similar to the better-known sum-of-squares polynomials. Seidler and de Wolff recently showed that sums of nonnegative circuit polynomials (or SONC polynomials for short) can be used to compute global lower bounds (called SONC bounds) for polynomials in this manner in polynomial time, as long as the polynomial is bounded from below and its support satisfies a certain nondegeneracy assumption. The quality of the SONC bound depends on the circuits used in the computation, but finding the set of circuits that yield the best attainable SONC bound among the astronomical number of candidate circuits is a non-trivial task that has not been addressed so far. In this paper we propose an efficient method to compute the optimal SONC lower bound by iteratively identifying the optimal circuits to use in the SONC bounding process. The method is based on a new proof of a recent result by Wang which states that (under the same nondegeneracy assumption) every SONC polynomial decomposes into SONC polynomials on the same support. Our proof, based on convex programming duality, removes the nondegeneracy assumption in Wang's result and motivates a column generation approach that generates an optimal set of circuits and computes the corresponding SONC bound in a manner that is particularly attractive for sparse polynomials of high degree and with a large number of unknowns. The method is implemented and tested on a large set of sparse polynomial optimization problems with up to 40 unknowns, of degree up to 60, and up to 3000 monomials in the support. The results indicate that the method is efficient in practice, requiring only a small number of iterations to identify the optimal circuits, with running times well under a minute for most of the instances and under 1.5 hours for the largest ones. Somewhat surprisingly, in the first set of the instances considered, the best SONC bound was very close to the best local minimum found using multi-start local minimization, showing both that the best local minima are close to global, and that the best attainable SONC lower bound is close to the best attainable lower bound.
\end{abstract}

\section{Introduction}

Polynomial optimization, that is, computing the infimum of a polynomial over a basic closed semialgebraic set is a fundamental computational problem in algebraic geometry with a wide range of applications such as discrete geometry \cite{BachocVallentin2008, BallingerBlekhermanCohnGiansiracusaKellySchurmann2009}, nonlinear dynamical systems \cite{GoluskinFantuzzi2019, KuntzThomasStanBarahona2019}, control \cite{HenrionGarulli2005,AylwardItaniParrilo2007,AhmadiMajumdar2016,DeitsTedrake2015}, extremal combinatorics \cite{RaymondSinghThomas2015, BarakEtal2019}, power systems engineering \cite{JoszMaeghtPanciaticiGilbert2015,GhaddarMarecekMevissen2016}, and statistics \cite{Papp2012}, to name a few. It is well-known to be an intractable problem; its difficulty stems from the computational complexity of deciding whether a given polynomial is nonnegative (either over $\R^n$ or over a semialgebraic set given by a list of polynomial inequalities) \cite{DickinsonGijben2014, BlekhermanParriloThomas2013}. The same problem, coupled with the additional, even more challenging, task of finding a rigorous certificate of nonnegativity (that is verifiable in polynomial time in exact arithmetic) is also a central question in symbolic computation and automated theorem proving \cite{Harrison2007}.

Practically scalable approaches to polynomial optimization rely on tractable approximations of cones of nonnegative polynomials. \emph{Inner} approximations based on easily verifiable sufficient conditions of nonnegativity are particularly desirable, as they can yield certificates of nonnegativity or rigorous lower bounds on the infimum, even if one can only compute approximately optimal (but feasible) numerical solutions to the optimization problems solved in the process of generating rigorous certificates (e.g., in hybrid symbolic-numerical methods). Undoubtedly, the most successful of these approximations to date has been sum-of-squares (SOS) cones, which date back to at least the early 2000s (see \cite{Nesterov2000, ParriloThesis2000, Lasserre2001}, and even the earlier work of Shor \cite{Shor1987}) and have given rise to polynomial optimization software such as GloptiPoly~3 \cite{gloptipoly} and SOSTOOLS \cite{sostools}.

More recently, a number of alternatives to SOS have been proposed to address difficulties often encountered when using SOS techniques for polynomials with either a large number of unknowns or a high degree. Alternatives such as DSOS and SDSOS polynomials \cite{AhmadiMajumdar2014, KuangEtal2017} are more tractable inner approximations of the SOS cones (using linear or second order cone programming in place of semidefinite programming); Ghasemi and Marshall suggests an approach using geometric programming \cite{GhasemiMarshall2012, GhasemiMarshall2013}. Other alternatives, such as sums of nonnegative (SONC) polynomials \cite{IlimanDeWolff2016} and sums of AM/GM exponentials (SAGE) \cite{ChandrasekaranShah2016} are different subcones of nonnegative polynomials that neither contain SOS cones nor are they contained by them, and thus, in principle have the potential to provide better bounds than SOS while being faster than SOS \cite{SeidlerDeWolff2018}. In this work, we focus on SONC polynomials, specifically on the problem of computing optimal SONC lower bounds efficiently.

The contributions of this paper are the following. First, in Section \ref{sec:prelim}, we provide a new conic optimization formulation for determining whether a polynomial is SONC; this formulation is somewhat smaller and simpler than the relative entropy programming formulation used in previous work on SONC polynomials. Using this formulation and convex programming duality, we prove in Section \ref{sec:support} that every SONC polynomial $f$ can be written as a sum of nonnegative circuit polynomials supported on the support of $f$ and a sum of monomial squares. This was also shown recently under mild conditions by Wang \cite{Wang2019}, using different methods. In Section \ref{sec:algorithm} we propose an algorithm, motivated by our proof of this result, to iteratively identify the circuits that appear in the optimal SONC decomposition. An implementation of this approach is discussed in Sections \ref{sec:implementation} and \ref{sec:numericalexperiments}, where we demonstrate that the approach can be used to find the optimal SONC lower bound on sparse polynomials with up to 3000 monomials in minutes. We conclude with a discussion regarding possible extensions and open questions in \mbox{Section \ref{sec:discussion}}.

\section{Preliminaries}\label{sec:prelim}

Recall the following notation and definitions. For vectors $\vz$ and $\vaf$ of dimension $n$, $\vz^\vaf$ is a shorthand for the monomial $\prod_{i=1}^n z_i^{\alpha_i}$. (Contrary to convention, we shall use $\vz$ to denote the unknowns of polynomials in spite of the unknowns being real, in order to avoid any confusion with the primal variables $\vx$ of our main optimization model.) For an $n$-variate polynomial $f$ given by $f(\vz) = \sum_{\vaf\in\supp(f)} f_{\vaf} \vz^{\vaf}$, the set of exponents $\supp(f)$ is called the \emph{support} of $f$. The \emph{Newton polytope} of $f$ is $\New(f) \defeq \conv(\supp(f))$, the closed convex hull of the support. A polynomial is a \emph{monomial square} if it can be written as $c \vz^\vaf$ with $c>0$ and $\vaf\in(2\N)^n$.

Following \cite{IlimanDeWolff2016}, we say that a polynomial $f$ is a \emph{circuit polynomial} if its support can be written as $\supp(f) = \{\vaf_1,\dots,\vaf_r,\vbeta\}$ such that the set $\{\vaf_1,\dots,\vaf_r\}$ is affinely independent and $\vbeta = \sum_{i=1}^r \lambda_i \vaf_i$ with some $\lambda_i>0$ satisfying $\sum_{i=1}^r \lambda_i = 1$. In other words, the scalars $\lambda_i$ are barycentric coordinates of the exponent $\vbeta$, which lies in the convex hull of the $\vaf_i$. The affine independence condition on the exponents implies that the barycentric coordinates $\lambda_i$ are unique and strictly positive.

The support set of a circuit polynomial is called a \emph{circuit}. The exponent $\vbeta$ is referred to as the \emph{inner exponent} of the circuit, while the $\vaf_i$ are the \emph{outer exponents}. Given a circuit $C$, $NC(C)$ denotes the set of nonnegative circuit polynomials supported on $C$. The vector of barycentric coordinates of the inner exponent is denoted by $\vlam(C)$.

Our starting point is the well-known characterization of nonnegative circuit polynomials \cite{IlimanDeWolff2016}:
\begin{proposition}\label{thm:sonc-main}
Let $f$ be an $n$-variate circuit polynomial satisfying $f(\vz) = \sum_{i=1}^r \fai \vz^{\vaf_i} + f_{\vbeta} \vz^{\vbeta}$ for some real coefficients $\fai$ and $f_\vbeta$ and suppose that $\vbeta = \sum_{i=1}^r \lambda_i \vaf_i$ with some $\lambda_i>0$ satisfying $\sum_{i=1}^r \lambda_i = 1$. Then $f$ is nonnegative if and only if $\vaf_i \in (2\N)^n$ and $\fai > 0$ for each $i$, and at least one of the following two alternatives holds:
\begin{enumerate}
\item $\vbeta \in (2\N)^n$ and $f_{\vbeta} \geq 0$, or
\item $|f_{\vbeta}| \leq \prod_{i=1}^r \left(\frac{\fai}{\lambda_i}\right)^{\lambda_i}$.
\end{enumerate}
\end{proposition}

It has been shown in \cite{DresslerIlimanDeWolff2017} that the second alternative in Proposition \ref{thm:sonc-main} is convex in the coefficients of $f$, moreover, it can be represented using $O(r)$ number of affine and relative entropy cone constraints. In this work, we use conic constraints involving the generalized power cone and its dual to represent nonnegative circuit polynomials, which has the advantage of requiring only one cone constraint per circuit.

The \emph{(generalized) power cone} with \emph{signature} $\vlam = (\lambda_1,\dots,\lambda_r)$ is the convex cone defined as
\begin{equation}\label{eq:power-cone-def}
    \cP_\vlam \defeq \left\{ (\vv,z)\in \R_+^r \times \R \,\middle|\, |z|\leq \vv^\vlam \right\}.
\end{equation}
It can be shown that $\cP_\vlam$ is a proper (closed, pointed, full-dimensional) convex cone for every $\vlam\in ]0,1[^r$, and that its dual cone (with respect to the standard inner product) is the following \cite{Chares2009}:
\[ \cP^*_\vlam \defeq \left\{ (\vv,z)\in \R_+^r \times \R \,\middle|\, |z|\leq \prod_{i=1}^r\left(\frac{v_i}{\lambda_i}\right)^{\lambda_i} \right\}. \]
This means that the second alternative in Proposition \ref{thm:sonc-main} can be written simply as a single cone constraint (and without additional auxiliary variables):
\begin{equation}\label{eq:NC}
 |f_{\vbeta}| \leq \prod_{i=1}^r \left(\frac{\fai}{\lambda_i}\right)^{\lambda_i}\; \Longleftrightarrow \; \big((f_{\vaf_1},\dots,f_{\vaf_r}),f_{\vbeta}\big) \in \cP^*_\vlam.
\end{equation}
Note that the cone depends on the circuit $C = \{\vaf_1,\dots,\vaf_r,\vbeta\}$ only through its signature $\vlam(C)$.

We say that a polynomial is a \emph{sum of nonnegative circuit polynomials}, or \emph{SONC} for short, if it can be written as a sum of monomial squares and nonnegative circuit polynomials. SONC polynomials are obviously nonnegative by definition. Since the nonnegativity of a circuit polynomial can be easily verified using Proposition \ref{thm:sonc-main}, the nonnegativity of a SONC polynomial can be certified by providing an explicit representation of the polynomial as a sum of monomial squares and nonnegative circuit polynomials. Such a certificate is called a \emph{SONC decomposition}. As long as the number of circuits is sufficiently small, a SONC decomposition can be verified efficiently. From (the conic version of) Carath{\'e}odory's theorem \cite[Corollary 17.1.2]{Rockafellar1970} it is clear that every SONC polynomial $f$ can be written as a sum of at most $|\supp(f)|$ nonnegative circuit polynomials, therefore, a ``short'' SONC decomposition exists. However, the number of circuits supported on the Newton polytope of a polynomial can be astronomical even for polynomials with a relatively small support set (see also Example \ref{ex:many_circuits}), and it is not clear which of these circuits will be needed in a SONC decomposition. This motivates the search for algorithms that can identify the relevant circuits and compute short SONC decompositions.

Suppose we are given a polynomial $f(\vz) = \sum_{\vaf\in\supp(f)} f_\vaf\vz^\vaf$ by its support and its coefficients in the monomial basis, and that we are given a set of circuits $\CC = \{C^1,\dots,C^N\}$. We shall assume, without loss of generality, that $\vzero \in \supp(f)$ and that $\supp(f) \subseteq \bigcup_{j=1}^N C^j$.

Let $\SS(\CC)$ be the set of polynomials that can be written as a sum of nonnegative circuit polynomials whose support is a circuit belonging to $\CC$ and monomial squares supported on $\supp(f)$. Using \mbox{Proposition \ref{thm:sonc-main}} and Equation \eqref{eq:NC}, one may see that deciding whether $f$ belongs to $\SS(\CC)$ amounts to solving a conic optimization (feasibility) problem. %(The condition that the circuits cover $\supp(f)$ is without loss of generality.)
We shall give the details of this optimization problem next. For theoretical reasons that will become clear later, we formulate this feasibility problem as a slightly more complicated optimization problem than what may appear necessary, with a carefully chosen objective function. As we shall see later in this section (\mbox{Lemma \ref{thm:strong-duality}}), this form guarantees that strong duality holds for this representation, with attainment in both the primal and the dual.

%%%% THE PARAGRAPH BELOW HAS EVERYTHING SET UP AS A VERY SIMPLE FEASIBILITY PROBLEM, WITHOUT MONOMIAL SQUARES
\deletethis{
Let $\vf$ denote the vector $(f_\vaf)_{\vaf\in\supp(f)}$ For each circuit $C^j$, denote the exponent vectors of its outer monomials by $\vaf^j_1,\dots,\vaf^j_{r_j}$ and the exponent vector of its inner monomial by $\vbeta^j$. Let $\vlam^j$ be the unique vector satisfying $\vbeta^j = \sum_{i=1}^{r_j}\lambda^j_i\vaf^j_i$. Finally, let $\vA^j \in \{0,1\}^{|\supp(f)|\times (r_j+1)}$ be the matrix whose $(k,\ell)$-th element is $1$ if the $k$-th element of $\supp(f)$ is the $\ell$-th element of the support of $C^j$, and $0$ otherwise. Now, it is immediate that $f$ has a SONC decomposition using the circuits $C^1,\dots,C^N$ if and only if the optimization (feasibility) problem below (whose decision variables $\vz_j \in \R^{r_j+1}$ $(j=1,\dots,N)$ are interpreted as the coefficient vectors of the nonnegative circuit polynomials we are seeking) has a solution:
\begin{equation}\label{eq:power}
\begin{alignedat}{2}
\underset{\vz_1,\dots,\vz_N}{\text{minimize}}\;\;\;\, & 0 \\
\text{subject to}\;\; & \sum_{j=1}^N \vA^j\vz_j = \vf\\
                      & \vz_j \in \cP^*_{\vlam^j} \quad j=1,\dots,N.
\end{alignedat}
\end{equation}

It is straightforward to extend \eqref{eq:power} to also include monomial squares in the decomposition using no more than $\supp(f)$ additional decision variables (or, equivalently, by simply changing the equality constraints corresponding to exponent vectors $\vaf\in(2\N)^n$) to a ``$\leq$'' inequality constraint, but for notational clarity this is left out of the model above.
}

Let $V$ be the vertices of $\New(f)$, and consider the following optimization problem, whose decision variables are indexed by $V$:
\begin{equation}\label{eq:existence}
\begin{alignedat}{2}
\underset{\vgamma\in\R_+^V}{\text{minimize}}\;\;\;\, & \sum_{\vaf\in V}\gamma_\vaf \\
\text{subject to}\;\; & (\vz\mapsto f(\vz) + \sum_{\vaf\in V}\gamma_\vaf \vz^\vaf) \in \SS(\CC).
\end{alignedat}
\end{equation}
%The decision variables are indexed by $V$; the precise interpretation of the main constraint is that the polynomial $\vz \mapsto f(\vz) + \sum_{\vaf\in V}\gamma_\vaf\vz^\vaf$ can be written as a sum of monomial squares supported on $\supp(f)$ (which we shall assume, without loss of generality, contains $\vzero$) and nonnegative circuit polynomials over the circuits $C^1,\dots,C^N$.

It is immediate that $f$ has a desired SONC decomposition if and only if the optimal objective function value of this problem is $0$ and if this infimum is attained.

Making the SONC decomposition of the polynomial in the constraint explicit, problem \eqref{eq:existence} can also be written as follows:
\begin{equation}\label{eq:existence-B}
\begin{alignedat}{2}
\text{minimize}\;\;\;\, & \sum_{\vaf\in V}\gamma_\vaf \\
\text{subject to}\;\; & f(\vz) + \sum_{\vaf\in V}\gamma_\vaf \vz^\vaf \equiv \sum_{j=1}^N p_j(\vz) + \sum_{\vaf\in\supp(f)\cap(2\N)^n} \delta_\vaf\vz^\vaf\\
& p_j \in NC(C^j)\ \qquad \quad j=1,\dots,N\\
& \gamma_\vaf \geq 0 \qquad \qquad \qquad \vaf \in V\\
& \delta_\vaf \geq 0 \qquad \qquad \qquad\vaf\in\supp(f)\cap(2\N)^n,
\end{alignedat}
\end{equation}
In computation, the polynomials required to be identical (by the first constraint) need to be represented by their coefficients in some basis, reducing the constraint to a system of $|\supp(f)|$ linear equations. It is convenient to use the monomial basis, in which case, by way of Proposition \ref{thm:sonc-main} and Eq.~\eqref{eq:NC}, the cone constraints $p_j\in NC(C^j)$ can be written as cone constraints involving $\cP^*_{\vlam(C^j)}$. The details of this formulation are given next; they are straightforward, but in order to write the formulation out explicitly, we need to introduce some additional notation.
%Let $\Vbar\defeq \supp(f)\setminus V$ and partition $\Vbar$ into $\Vbar_\text{even}\defeq\Vbar\cap(2\N)^n$ and $\Vbar_\text{odd}\defeq\Vbar\setminus(2\N)^n$.

Let us partition $\supp(f)$ into $S_\text{even}\defeq\supp(f)\cap(2\N)^n$ and $S_\text{odd}\defeq\supp(f)\setminus(2\N)^n$.
Now, $f(\vz) + \sum_{\vaf\in V}\gamma_\vaf\vz^\vaf$ is SONC if and only if there exist 
nonnegative circuit polynomials $p_1,\dots,p_N$ supported on $C^1,\dots,C^N$, respectively and coefficients $\delta_\vaf \geq 0$ for each $\vaf\in S_\text{even}$ such that $p_1(\vz)+\dots+p_N(\vz) + \sum_{\vaf\in S_\text{even}} \delta_\vaf\vz^\vaf = f(\vz) + \sum_{\vaf\in V}\gamma_\vaf\vz^\vaf$.

%Let $\vf_\text{even}\defeq(f_\vaf)_{\vaf\inS_\text{even}}$ and similarly let $\vf_\text{odd} \defeq (f_\vaf)_{\vaf\inS_\text{odd}}$.
%For each circuit $C^j$, denote the exponent vectors of its outer monomials by $\vaf^j_1,\dots,\vaf^j_{r_j}$ and the exponent vector of its inner monomial by $\vbeta^j$. Let $\vlam^j$ be the unique vector satisfying $\vbeta^j = \sum_{i=1}^{r_j}\lambda^j_i\vaf^j_i$.
For each $j\in\{1,\dots,N\}$, let $\vA^j \in \{0,1\}^{|\supp(f)|\times (r_j+1)}$ be the matrix whose \mbox{$(k,\ell)$-th} element is $1$ if the $k$-th element of $\supp(f)$ is the $\ell$-th element of the support of $C^j$, and $0$ otherwise. %Finally, let $\mathbf{1}_V \in \{0,1\}^{S_\text{even}}$ be the column vector whose $k$-th component is $1$ if the $k$-th element of $S_\text{even}$ is in $V$ and 0 otherwise.
In what follows, $\vA^j_{\vaf,\cdot}$ denotes the row of $\vA^j$ indexed by the exponent vector $\vaf$. Noting that $V\subseteq S_\text{even}$, we can now write the optimization problem \eqref{eq:existence-B} in the monomial basis as follows:
\begin{equation}\label{eq:power}
\begin{alignedat}{2}
\underset{\vgamma,\vx_1,\dots,\vx_N}{\text{minimize}}\;\;\;\, & \sum_{\vaf\in V}\gamma_\vaf \\
\text{subject to}\;\; & \sum_{j=1}^N \vA^j_{\vaf,\cdot}\vx_j - \gamma_\vaf  \leq f_\vaf \qquad \vaf\in V\\
& \sum_{j=1}^N \vA^j_{\vaf,\cdot}\vx_j \leq f_\vaf \qquad \vaf\in S_\text{even} \setminus V\\
& \sum_{j=1}^N \vA^j_{\vaf,\cdot}\vx_j = f_\vaf \qquad \vaf\in S_\text{odd}\\
& \gamma_\vaf \geq 0 \qquad \vaf \in V,\\
& \vx_j \in \cP^*_{\vlam(C^j)} \quad j=1,\dots,N.
\end{alignedat}
\end{equation}
To see this, note that the decision variable $\vx_j \in \R^{r_j+1}$ $(j=1,\dots,N)$ can be interpreted as the coefficient vector of the nonnegative circuit polynomial $p_j$ supported on $C^j$, $\vA^j_{\vaf,\cdot} \vx_j$ is the coefficient of $\vz^\vaf$ in $p_j(\vz)$, and the interpretation of the linear constraints is that $\sum_{j=1}^n p_j(\cdot) = f(\cdot)+\sum_{\vaf\in V}\gamma_\vaf (\cdot)^\vaf - \sum_{\vaf\in S_\text{even}} \delta_\vaf (\cdot)^\vaf$ for some nonnegative coefficients $\delta_\vaf$ ($\vaf\in S_\text{even}$) whose values are the slacks of the first two sets of inequality constraints.

In the dual problem of \eqref{eq:power}, the components of the vector of decision variables $\vy$ may be indexed by monomials in $\supp(f) = V \cup (S_\text{even}\setminus V) \cup S_\text{odd}$, and the dual optimization problem can be written as follows:
\begin{equation}\label{eq:power-dual-A}
\begin{alignedat}{2}
\underset{\vy\in\R^{\supp(f)}}{\text{maximize}}\;\;\;\, & \vf^\T\vy \\
\text{subject to}\;\;\, & -(\vA^j)^\T\vy \in \cP_{\vlam(C^j)} \qquad j=1,\dots,N\\
                        & 1 + y_\vaf \geq 0 \qquad \vaf\in V\\
                        & y_\vaf \leq 0 \qquad\quad\;\;\, \vaf\in S_\text{even}.
\end{alignedat}
\end{equation}
The constraints in \eqref{eq:power-dual-A} can be further simplified. Recalling the definition of $\vA^j$, we have that $- (\vA^j)^\T\vy = (-y_\vaf)_{\vaf\in C^j}$.\deletethis{ if we define $y_\vaf = -1$ for every $\vaf\in V$.} It is also convenient to replace in notation $\vy$ with $-\vy$ throughout. %Thus, \eqref{eq:power-dual-A} also be written in the following manner:
%\begin{equation}\label{eq:power-dual}
%\begin{alignedat}{2}
%\underset{\vy}{\text{maximize}}\;\;\;\, & \vf^\T\vy \\
%\text{subject to}\;\;\, &  (-y_\vaf)_{\vaf\in C^j} \in P_{\vlam^j} \quad j=1,\dots,N\\
%                        & (-y_\vaf)_{\vaf\in S_\text{even}} \geq \vzero, \quad \sum_{\vaf\in S_\text{even}} (-y_\vaf) \leq 1.
%\end{alignedat}
%\end{equation}
This leads to the following representation of the dual of \eqref{eq:existence}:
\begin{equation}\label{eq:power-dual}
\begin{alignedat}{2}
\underset{\vy\in\R^{\supp(f)}}{\text{maximize}}\;\;\;\, & -\vf^\T\vy \\
\text{subject to}\;\;\,  &  (y_\vaf)_{\vaf\in C^j} \in \cP_{\vlam(C^j)} \quad j=1,\dots,N\\
                         &  y_\vaf \geq 0 \qquad \vaf\in \supp(f)\cap(2\N)^n,\\
                         &  y_\vaf \leq 1 \qquad \vaf\in V.
\end{alignedat}
\end{equation}

We are now ready to show that all these problems have attained optimal values, and that strong duality holds for the optimization problems in Eq.~\eqref{eq:existence} and Eq.~\eqref{eq:power-dual}.

\begin{lemma}\label{thm:strong-duality}
Suppose that $V \subseteq (2\N)^n$ and that for every $\vaf_j \in \supp(f) \setminus V$ there is a circuit $C \in \CC$ whose inner monomial is $\vaf_j$ and whose outer monomials are all members of $V$.
Then the optimization problem \eqref{eq:existence-B} has a strictly feasible solution as well as an optimal solution. Therefore, both \eqref{eq:existence} and \eqref{eq:power-dual} have optimal solutions, and the optimal objective function values are equal.
\end{lemma}

\begin{proof}
We can construct a strictly feasible solution to \eqref{eq:existence-B} as follows.
First, we fix $\delta_\vaf = 1$ for each $\vaf\in \supp(f)\cap(2\N)^n$.
Second, by assumption, for each exponent $\vaf_j \in \supp(f) \setminus V$ we can find a nonnegative circuit polynomial $p_j \in NC(C^j)$ whose inner monomial has the coefficient $f_{\vaf_j}$ (if ${\vaf_j}\in \supp(f)\cap(2\N)^n$) or $f_{\vaf_j}-1$ (if ${\vaf_j}\in \supp(f)\cap(2\N)^n$), while its outer monomials have sufficiently large positive coefficients to ensure that $p_j$ is in the interior of the $NC(C^j)$. In the resulting sum $p(\vz)\defeq\sum_j p_j(\vz) + \sum_\vaf \delta_\vaf z^\vaf$, the coefficient of each $\vz^\vaf$ for $\vaf\in \supp(f)\setminus V$ is equal to $f_\vaf$. By further increasing the outer coefficients in each $p_j$, we can also ensure that for each $\vaf\in V$ the coefficient of each $\vz^\vaf$ in $p$ is strictly greater then $f_\vaf$. The resulting $p$ is a strictly feasible solution; we can set each $\gamma_\vaf$ to an appropriate positive value to equate the two sides of the first constraint of \eqref{eq:existence-B}.

Thus, the minimization problem \eqref{eq:existence-B} is feasible; however it cannot be unbounded since the objective function is constrained to be nonnegative on the feasible region. Therefore, the infimum is finite.

To see that this finite optimal value is attained, observe that because each $\gamma_\vaf$ is nonnegative and because there is some finite objective function value $\Gamma$ attained by the strictly feasible solution exhibited above, we can add to the formulation \eqref{eq:existence-B} the redundant constraints $\gamma_\vaf \in [0,\Gamma]$ for every $\vaf \in V$. Then, since each $\gamma_\vaf$ is bounded, and each of the polynomials $p_j$ and $\delta_\vaf \vz^\vaf$ on the right-hand side of \eqref{eq:existence-B} the first constraint of is a nonnegative polynomial, every norm $\|\cdot\|$ of each $p_j$ and $\delta_\vaf$ can also be bounded a priori by $\|f\|+\sum_{\vaf \in V} \gamma_\vaf \|\vz^\vaf\|$. Thus, the feasible set is compact, and the infimum in \eqref{eq:existence-B} is attained using the Weierstrass Extreme Value Theorem.

We have shown that \eqref{eq:existence-B} has an optimal solution and a Slater point. This implies that \eqref{eq:existence-B} and its dual have optimal solutions with the same objective function values, therefore the equivalent problems \eqref{eq:existence} and its dual \eqref{eq:power-dual} also have optimal solutions with the same optimal objective function value.
\end{proof}

The number of decision variables in the explicit conic formulation \eqref{eq:power}, which can be directly fed to a conic optimization solver, is $|V|+\sum_{j=1}^N (r_j+1)$. This can be prohibitively large for practical computations if the number of circuits $N$ is large. This motivates the rest of the paper, where we narrow down the set of circuits that may be needed in a SONC decomposition and provide an algorithm to iterative identify the useful circuits.

\section{Support of SONC decompositions}\label{sec:support}

Let $f(\vz) = \sum_{\vaf\in\supp(f)} f_\vaf\vz^\vaf$ be a SONC polynomial. It is straightforward to argue that in every SONC decomposition of $f$, every circuit polynomial must be supported on a subset of $\New(f)$; for completeness, we include a short argument in the proof of Theorem \ref{thm:support} below. It is equally natural to ask whether there exists a SONC decomposition for $f$ in which every circuit polynomial is supported on a subset of $\supp(f)$. That this is indeed true was first shown recently in \cite{Wang2019} using combinatorial and algebraic techniques (and some assumptions on the structure of the support); we shall provide an independent proof using convex programming duality (without any assumptions). In the proof, which also motivates the algorithmic approach of the next section, we will need the following simple lemma.

%\begin{lemma}\label{thm:extremecombinations}
%	Let $\vaf_1,\dots,\vaf_N$ and $\vbeta$ be given vectors in $\R^n$, and consider the convex polytope $P$ consisting of convex combinations of the $\vaf_i$ that yield $\vbeta$:
%	\[ P = \left\{\vlam\in\R_+^N\,\middle|\, \sum_{i=1}^N \lambda_i\vaf_i = \vbeta \text{ and } \sum_{i=1}^N \lambda_i = 1 \right\}. \]
%	Then for every extreme point $\vlam$ of $P$, the set $\{\vaf_i\,|\,\lambda_i > 0\}$ is affinely independent.
%\end{lemma}

%An immediate corollary is the following.
%\begin{lemma}\label{thm:extremeinequalities}
%Suppose that $\vc\in\R^N$ and $d\in\R$ are chosen such that the inequality
%\begin{equation}\label{eq:lemmaineq}
%\vc^\T\vlam \leq d
%\end{equation}
%holds for every $\vlam \in P$ such that $\{\vaf_i\,|\,\lambda_i > 0\}$ is affinely independent. Then \eqref{eq:lemmaineq} holds for every %$\vlam \in P$.
%\end{lemma}

\begin{lemma}\label{thm:extremeinequalities}
	Let $\vc\in\R^N$ and $d\in\R$ be arbitrary. Furthermore, let $\vaf_1,\dots,\vaf_N$ and $\vbeta$ be given vectors in $\R^n$, and consider the convex polytope $P$ consisting of all convex combinations of the $\vaf_i$ that yield $\vbeta$:
	\[ P = \left\{\vlam\in\R_+^N\,\middle|\, \sum_{i=1}^N \lambda_i\vaf_i = \vbeta \text{ and } \sum_{i=1}^N \lambda_i = 1 \right\}. \]
	Then, if the inequality
		\begin{equation}\label{eq:lemmaineq}
	\vc^\T\vlam \leq d
	\end{equation}
	holds for every $\vlam \in P$ for which the set $S_\vlam \defeq \{\vaf_i\,|\,\lambda_i > 0\}$ is affinely independent, then \eqref{eq:lemmaineq} holds for every $\vlam \in P$.
\end{lemma}

\begin{proof}
	This is a reformulation of the statement that every extreme point $\vlam$ of the convex polytope $P$ corresponds to an affinely independent $S_\vlam$. This is immediate from the theory of linear optimization: the basic components of every basic feasible solution of the (feasibility) linear optimization problem
	\begin{equation*}
	\begin{alignedat}{2}
	\text{find}_\vlam\;\;\;& \sum_{i=1}^N \lambda_{i} \vaf_i = \vbeta\\
	& \sum_{i=1}^N \lambda_{_i} = 1\\
	& \lambda_i \geq 0\qquad i=1,\dots,N
	\end{alignedat}
	\end{equation*}
	correspond to linearly independent ($(n+1)$-dimensional) vectors from $\{\binom{\vaf_1}{1}, \dots, \binom{\vaf_N}{1}\}$; thus, the nonzero components of every vertex of $P$ correspond to affinely independent $S_\vlam$.
\end{proof}

\begin{theorem}\label{thm:support}
Every SONC polynomial $f$ has a SONC decomposition in which every nonnegative circuit polynomial and monomial square is supported on a subset of $\supp(f)$.
\end{theorem}
\begin{proof}
First, we argue that no monomial outside the Newton polytope $\New(f)$ can appear in any SONC decomposition. Suppose otherwise, then the convex hull of the union of the circuits is a convex polytope that has an extreme point $\vaf\not\in\New(f)$. The corresponding monomial $\vz^\vaf$ has a 0 coefficient in $f$. At the same time, $\vz^\vaf$ can only appear in the SONC decomposition as a monomial square or as an outer monomial in a circuit, but never as an inner monomial. Therefore, its coefficient is 0 only if its coefficient is 0 in every circuit it appears in, which is a contradiction.

%Now suppose that contrary to the statement of the theorem, there exists a SONC decomposition of $f$ that involves a nonnegative circuit polynomial supported on a circuit that is not a subset of $\supp(f)$, but that $f$ has no SONC decomposition supported on $\sonc(p)$. We shall arrive at a contradiction by showing that the optimal solutions of 

Next, consider two instances of problem \eqref{eq:existence}, or equivalently \eqref{eq:power}: in the first instance, to be called $(P_1)$, we choose the circuits $\CC = \{C^1,\dots,C^N
\}$ to be the set of all circuits that are subsets of $\supp(f)$, while in the second one, $(P_2)$, we choose the circuits to be the set of all circuits that are subsets of $\New(f)\cap\N^n$. Let the duals of the corresponding problems, written in the form \eqref{eq:power-dual}, be $(D_1)$ and $(D_2)$. According to the discussion around \eqref{eq:existence}, it suffices to show that $(P_1)$ and $(P_2)$ have the same optimal objective function values, since in that case either both problems have an optimal solution attaining the value $0$ (and thus SONC decompositions using both sets of circuits exist) or both problems have a strictly positive optimal value (and thus no SONC decomposition exists using either set of circuits).

Using Lemma \ref{thm:strong-duality}, $(P_1)$ and $(D_1)$ have optimal solutions $(\vx_1^*,\dots,\vx_N^*)$ and $\vy^*$ attaining equal objective function values. We now use these solutions to construct feasible solutions for both $(P_2)$ and $(D_2)$ that attain the same objective function value.

For $(P_2)$ this is straightforward: in the formulation \eqref{eq:power}, keep the coefficients $\vx_j$ of the circuit polynomials appearing in $(P_1)$ the same value $\vx_j^*$, and set $\vx_j=\vzero$ for every new circuit that appears only in $(P_2)$.

For $(D_2)$, we also keep $\vy_\vaf = \vy^*_\vaf$ for every $\vaf\in\supp(f)$. With this choice, regardless of the choice of the remaining components of $\vy$, every constraint in $(D_2)$ that already appeared in $(D_1)$ is automatically satisfied; moreover, the objective function remains unchanged, since $f_\vaf=0$ for the new variables. Therefore, it only remains to show that $(\vy_\vaf)_{\vaf\in(\New(f)\cap\N^n)\setminus\supp(f)}$ can be chosen in a way that every cone constraint $(y_\vaf)_{\vaf\in C} \in \cP_{\vlam(C)}$ corresponding to a circuit $C$ supported on $\New(f)\cap\N^n$ is satisfied. We show, constructively, a slightly stronger statement: that if we assign values to the new components of $\vy$ one-by-one in any order, at each step it is possible to assign a value to the component at hand in a way that satisfies every conic inequality that only involves already processed exponents.

Suppose that some exponents have been given consistent values and  let $\hat\vaf \in (\New(f)\cap\N^n)\setminus\supp(f)$ be the exponent whose corresponding $y_{\hat\vaf}$ needs to be assigned a value next. In every circuit that it appears in, the exponent $\hat\vaf$ is either an inner exponent, in which case the cone constraint only provides an upper bound on $|y_{\hat\vaf}|$, or an outer exponent, in which case the cone constraint only provides a lower bound on $y_{\hat\vaf}$. In particular, if $\hat\vaf \not\in (2\N)^n$, then it cannot be an outer exponent, and $y_{\hat\vaf} = 0$ will be a consistent choice. Similarly, if $\hat\vaf \in (2\N)^n$ but $\hat\vaf$ appears only as inner (respectively, outer) exponent in every circuit, then it is easy to find a consistent value for $y_{\hat\vaf}$. (Zero, or a sufficiently large positive value, respectively.) The only non-trivial case is when $\hat\vaf \in (2\N)^n$ and $\hat\vaf$ appears both as inner and as outer exponent in a circuit.

Let $C_1$ be one of the circuits in which $\hat\vaf$ is an inner exponent and which gives the lowest upper bound on $y_{\hat\vaf}$, and let $C_2$ be one of the circuits in which $\hat\vaf$ is an outer exponent and which gives the greatest lower bound on $y_{\hat\vaf}$. We need to show that these bounds are consistent. Let the outer exponents of the circuit $C_1$ be
 $\vaf_1,\dots\vaf_r$ and let $(\lambda_i)_{i=1,\dots,r}$ be the barycentric coordinates of $\hat\vaf$ in this circuit:
 \begin{equation}\label{eq:c1}
 \hat\vaf = \sum_{i=1}^r \lambda_i\vaf_i.
 \end{equation}
Similarly in circuit $C_2$, let $\veta$ be the inner exponent, let $\hat\vaf$ and $\vomega_1,\dots,\vomega_s$ be the outer exponents, and let $\vxi$ denote the barycentric coordinates of $\veta$:
 \begin{equation}\label{eq:c2}
 \veta = \xi_0 \hat\vaf + \sum_{j=1}^s \xi_j\vomega_j.
 \end{equation}
Then it suffices to show that there exists a $y_{\hat\vaf}>0$ such that
\begin{equation}\label{eq:log1}
\log(y_{\hat\vaf}) \leq \sum_{i=1}^r \lambda_i\log(y_{\vaf_i})
\end{equation}
to satisfy the cone constraint $|y_{\hat\vaf}| \leq \prod_{i=1}^r y_{\vaf_i}^{\lambda_i}$ from $C_1$ and
\begin{equation}\label{eq:log2}
\log(|y_\veta|) \leq \xi_0 \log(y_{\hat\vaf}) +  \sum_{j=1}^s \xi_j\log(y_{\vomega_j})
\end{equation}
to satisfy the cone constraint $|y_{\veta}| \leq y_{\hat{\vaf}}^{\xi_0}\prod_{j=1}^s y_{\vomega_j}^{\xi_j}$ from $C_2$. The inequalities \eqref{eq:log1} and \eqref{eq:log2} are consistent if and only if the lower and upper bounds they give for $\log(y_{\hat\vaf})$ are consistent, that is, if
\[ \frac{1}{\xi_0}( \log(|y_\veta| - \sum_{j=1}^s\xi_j\log(y_{\vomega_j})) \leq \sum_{i=1}^r \lambda_i\log(y_{\vaf_i}), \]
which can be rearranged as 
\begin{equation}\label{eq:tadaa} \log|y_\veta| \leq \sum_{i=1}^r \xi_0\lambda_i\log(y_{\vaf_i}) + \sum_{j=1}^s\xi_j\log(y_{\vomega_j}).
\end{equation}
Now, note that from \eqref{eq:c1} and \eqref{eq:c2} we also have
\[ \veta = \sum_{i=1}^r\xi_0\lambda_i \vaf_i + \sum_{j=1}^s \xi_j\vomega_j, \]
with coefficients $\xi_0\lambda_i \geq 0$ and $\xi_j \geq 0$ satisfying $\sum_{i=1}^r \xi_0\lambda_i + \sum_{j=1}^s\xi_j=1$. Thus, \eqref{eq:tadaa} is almost identical to a power cone inequality corresponding to a circuit. The only difference is that the ``outer exponents'' $\{\vaf_1,\dots,\vaf_r, \vomega_1,\dots,\vomega_s\}$ are not necessarily affinely independent, thus these exponents and $\veta$ do not form a circuit. (If they do, we are done, by the inductive assumption that all power cone constraints corresponding to circuits that consists of assigned components of $\vy$ are satisfied.) 

We can now invoke Lemma \ref{thm:extremeinequalities} with $\{\vaf_1,\dots,\vaf_r, \vomega_1,\dots,\vomega_s\}$ playing the role of $\vaf_1,\dots,\vaf_N$, the exponent vector $\veta$ playing the role of $\vbeta$, and $(\log(y_{\vaf_1}),\dots,\log(y_{\vaf_r}), \log(y_{\vomega_1}),\dots,\log(y_{\vomega_s}))$ playing the role of $\vc$, and $\log|y_\veta|$ playing the role of $d$: if every power cone inequality corresponding to a circuit with inner exponent $\veta$ holds, then \eqref{eq:tadaa} also holds. By the argument preceding \eqref{eq:tadaa}, this implies that $y_{\hat{\vaf}}$ can be assigned a value that is consistent with the values of all already processed component of $\vy$.
\end{proof}

%\begin{figure}[h]
%    \centering
%    \includegraphics[scale = 0.3]{circuit.png}
%    \includegraphics[scale = 0.3]{simplices.png}
%    \caption{Setup for Theorem \ref{thm:support}}
%    \label{fig:circuit}
%\end{figure}

\section{SONC bounds and circuit generation}\label{sec:algorithm}

Theorem \ref{thm:support} allows us to dramatically simplify the search for SONC decompositions when the polynomial to decompose is sparse, that is, when $\supp(f)$ is much smaller than $\New(f)\cap\N^n$. That said, even the number of circuits supported on $\supp(f)$ can be exponentially large in the number of variables as the following simple example shows.

\begin{example}\label{ex:many_circuits}
Let $\ve_i$ denote the $i$th unit vector and $\vone\defeq \sum_{i=1}^n\ve_i$, and let $\supp(f)$ be the set $\{\vzero, \vone, 2n\ve_1, \dots,2n\ve_n, 4n\ve_1,\dots,4n\ve_n\}$. This support set has only $2n+2$ elements, but it supports $2^n$ different circuits with $\vone$ as the inner exponent.
\end{example}

In this section, we present an iterative method to identify the circuits that are necessary in a SONC decomposition of a given polynomial $f$. We present the algorithm for the more general and widely applicable problem of finding \emph{the highest SONC lower bound} for a polynomial, which is defined as the negative of the optimal value of the optimization problem
\begin{equation}\label{eq:bound}
\begin{alignedat}{2}
\underset{\gamma\in\R}{\text{minimize}}\;\;\;\, & \gamma \\
\text{subject to}\;\; & (\vz\mapsto f(\vz) + \gamma) \in \SS(\CC)
\end{alignedat}
\end{equation}
This is a well-defined quantity for every polynomial $f$ that has a SONC decomposition, and by extension for every polynomial that has a SONC bound, as the following Lemma shows.
\begin{lemma}\label{thm:bound-attainment}
Suppose that $f$ has a SONC decomposition with a given set of circuits $\CC$. Then \eqref{eq:bound} attains a minimum.
\end{lemma}
\begin{proof}
The proof is essentially the same as the argument used in the last step of the proof of Lemma \ref{thm:strong-duality}. If $f$ is SONC, then $\gamma=0$ is a feasible solution to \eqref{eq:bound}. At the same time, the problem cannot be unbounded; indeed, the infimum cannot be lower than $-f(\vzero)$. So the infimum in \eqref{eq:bound} is finite. 
Moreover, problem \eqref{eq:bound} can be equivalently written as 
\begin{equation}\label{eq:bound-B}
\begin{alignedat}{2}
{\text{minimize}}\;\;\;\, & \gamma \\
\text{subject to}\;\; & f(\vz) + \gamma = \sum_{j=1}^N p_j(\vz) + \sum_{\vaf\in\supp(f)\cap(2\N)^n} \delta_\vaf\vz^\vaf\\
& p_j \in NC(C^j)\ \quad j=1,\dots,N\\
& \gamma \in [-f(\vzero), 0]\\
& \delta_\vaf \geq 0 \qquad \vaf\in\supp(f)\cap(2\N)^n
\end{alignedat}
\end{equation}
Since $\gamma$ is already bounded, and each of the polynomials $p_j$ and $\delta_\vaf \vz^\vaf$ on the right-hand side of the first constraint is a nonnegative polynomial, any norm of each $p_j$ and $\delta_\vaf$ can also be bounded a priori by the same norm of $f+\gamma$, and thus the feasible region of \eqref{eq:bound-B} is compact. The claim now follows from the Weierstrass Extreme Value Theorem.
\end{proof}

\deletethis{
\textcolor{blue}{
?? \eqref{eq:bound} does \emph{not} always have a feasible, let alone strictly feasible solution. But its dual trivially has a strictly feasible solution. Therefore, what we have is that \emph{if} $f$ admits a SONC lower bound, \emph{then} it has an attained minimum. (This is what we showed above, too, without referencing the dual. This argument would also show strong duality in a $\min=\sup$ sense.) One issue that it is unclear whether the dual has attainment. (Since we want to use this in a numerical method, it's ok, though. But it's inelegant.) We also have a missing step of just deciding whether a SONC bound exists. This is not a dealbreaker, but it's inelegant. Is there a formulation that has strong duality that we can use to decide whether a SONC bound \emph{exists}? 
}
}

We now consider the problem of identifying the circuits necessary to obtain the strongest possible SONC lower bound on a polynomial. Consider the optimal solution of \eqref{eq:bound} for a set of circuits $\CC = \{C^1,\dots,C^N\}$ for which this problem attains a minimum. Analogously to \eqref{eq:power-dual}, the dual of \eqref{eq:bound} can be written as
\begin{equation}\label{eq:bound-dual}
\begin{alignedat}{2}
\underset{\vy\in\R^{\supp(f)}}{\text{maximize}}\;\;\;\, & -\vf^\T\vy \\
\text{subject to}\;\;\,  &  (y_\vaf)_{\vaf\in C^j} \in \cP_{\vlam(C^j)} \quad j=1,\dots,N\\
                         &  y_\vaf \geq 0 \quad \vaf\in \supp(f)\cap(2\N)^n,\\
                         &  y_\vzero = 1.
\end{alignedat}
\end{equation}
Although Eq.~\eqref{eq:bound} does not always have a Slater point, its dual \eqref{eq:bound-dual} trivially has, therefore, the supremum in \eqref{eq:bound-dual} equals the attained minimum in \eqref{eq:bound}. Thus, an (approximately) optimal solution to \eqref{eq:bound-dual} serves as a certificate of (approximate) optimality of the bound given by \eqref{eq:bound} for the given set of circuits. For brevity, we state this formally without a proof.
\begin{lemma}\label{thm:bound-strong-duality}
For every polynomial $f$ and set of circuits $\CC=\{C^1,\dots,C^N\}$, the optimization problem \eqref{eq:bound-dual} has a Slater point. Therefore, if $f$ has a SONC lower bound, then the optimal value of \eqref{eq:bound-dual} equals the (attained) optimal values of \eqref{eq:bound} and \eqref{eq:bound-B}.
\end{lemma}

Applying this Lemma by substituting the set of all circuits supported on $\supp(f)$ for $\CC$, we have that if the optimal solution $\vy^*$ of \eqref{eq:bound-dual} satisfies
\begin{equation}\label{eq:constr}
(y^*_\vaf)_{\vaf\in C} \in \cP_{\vlam(C)}
\end{equation}
for every circuit $C$ supported on $\supp(f)$, then the optimal value $\gamma^*$ of \eqref{eq:bound} cannot be improved by adding more circuits supported on $\supp(f)$ to the problem. Conversely, if we can find a circuit $C$ supported on $\supp(f)$ for which \eqref{eq:constr} is violated, then adding $C$ to the set $\CC$ may improve the bound given by \eqref{eq:bound}. Finally, we can repeat the argument of Thm.~\ref{thm:support} (with the primal-dual optimization pair \eqref{eq:bound}-\eqref{eq:bound-dual} playing the role of \eqref{eq:existence-B}-\eqref{eq:power-dual}) to show that adding any circuits that are not supported on $\supp(f)$ to $\CC$ also cannot improve the bound.
 
This motivates the iterative algorithm shown in \mbox{Algorithm \ref{alg:CG}}.

%%% COPIED FROM OLD PAPER FOR FORMAT, UPDATE
\SetKwInput{KwParameters}{parameters}
\begin{algorithm2e}
%\KwParameters{$\varepsilon>0$}
initialize $\CC=\{C^1,\dots,C^N\}$ \;
\Repeat{false}{
    solve the primal-dual pair \eqref{eq:bound-B}-\eqref{eq:bound-dual} for the optimal $(\gamma^*,p^*,\delta^*)$  and $\vy^*$ \;
    find the circuit $C$ supported on $\supp(f)$ for which \eqref{eq:constr} is the most violated \nllabel{find-circuit-step} \;
    \eIf{no circuit violating \eqref{eq:constr} exists}
    {
    \Return $\gamma^*$ and the SONC decomposition $(p^*,\delta^*)$ of $f+\gamma^*$ \;
    }
    {
        add circuit $C$ found in Step \ref{find-circuit-step} (and possibly other circuits) to $\CC$ \nllabel{add-circuit-step}
    }
}
\caption{SONC bound with iterative circuit generation}\label{alg:CG}
\end{algorithm2e}

We defer the discussion on the initialization step to the end of this subsection and focus on the main loop first, assuming that the initial set of circuits $\CC$ has been chosen such that the optimal solutions sought in the first iteration exist.

The most violated constraint in Line \ref{find-circuit-step} can be efficiently computed using the following observation: for a fixed exponent vector $\vbeta$, finding the circuit corresponding to the most violated constraint among circuits with inner monomial $\vz^\vbeta$ amounts to solving the linear optimization problem
\begin{equation}\label{eq:CG_LP}
\begin{alignedat}{2}
{\text{minimize}}\;\;\;\;& \sum_{\vaf\in\supp(f)\setminus\{\vbeta\}} \lambda_\vaf \log(y_{\vaf})\\
\text{subject to}\;\;\;& \sum_\vaf \lambda_\vaf \vaf = \vbeta\\
& \sum_\vaf \lambda_\vaf = 1\\
& \lambda_\vaf \geq 0\qquad \forall\vaf \in \supp(f)\setminus\{\vbeta\}
\end{alignedat}
\end{equation}
Based on Lemma~\ref{thm:extremeinequalities}, every basic feasible solution $\vlam$ of \eqref{eq:CG_LP} corresponds to a circuit whose outer monomials are $\{z^\vaf\,|\,\lambda_\vaf > 0\}$ and whose inner monomial is $\vz^\vbeta$. Recalling the definition of the power cone from Eq.~\eqref{eq:power-cone-def}, if $\vlam^*$ is an optimal basic feasible solution of \eqref{eq:CG_LP} and the optimal value is $v^*$, then the inequality \eqref{eq:constr} corresponding to $\vlam^*$ (and the circuit $C$ determined by $\vlam^*$) is violated if and only if $\exp(v^*) < |y_\vbeta|$. Solving \eqref{eq:CG_LP} for each $\vbeta$, we can either conclude that there are no circuits to add to the formulation or find up to one promising circuit for each $\vbeta$ to add to the formulation in Line \ref{add-circuit-step}. In our implementation we add to $\CC$ the circuit corresponding to the most violated inequality for each $\vbeta$.

\paragraph{Initialization} Problem \eqref{eq:existence} and the proof of Lemma \ref{thm:strong-duality} suggest a strategy for the initialization step of Algorithm \ref{alg:CG}, which is also entirely analogous to solving linear optimization problems using a two-phase method. We can apply the same circuit generation strategy as above to an instance of problem \eqref{eq:existence}, where $\CC$ contains all possible circuits supported on $\supp(f)$. (Additionally, we may replace $f$ by any $f+c$ with an arbitrary constant $c$.) An initial set of circuits for which this optimization problem is feasible can be easily found: for each exponent $\vaf \in \supp(f)$ that is not a monomial square (that is, for which either $\vaf\not\in(2\N)^n$ or $f_\vaf < 0$ or both), find a circuit whose inner exponent is $\vaf$ and whose outer exponents are among the vertices $V$ of $\New(f)$. This can be done by computing a basic feasible solution of a linear feasibility problem with $|V|$ variables. If for any $\vaf$ such a circuit does not exist, then \eqref{eq:existence} trivially does not have a feasible solution, and $f+c$ does not have a SONC bound. On the other hand, if the initial set of circuits exists, then \eqref{eq:existence} can be solved using the same column generation strategy, and we either find that the optimal objective function value of \eqref{eq:existence} is positive, in which case $f+c$ does not have a SONC bound, or the optimal value is $0$. In the latter case we also obtain a feasible solution with a current set of circuits $\CC$. This set can be used as the initial set of circuits in Algorithm \ref{alg:CG} to find the best SONC bound on $f$.

\begin{remark}
	There are many polynomials for which Algorithm \ref{alg:CG} for the SONC bounding problem \eqref{eq:bound} can be trivially initialized, without using \eqref{eq:existence} as a ``Phase I'' problem as described above. A sufficient condition is the following: suppose that for every $\vaf\in\supp(f)$ for which $\vaf\not\in(2\N)^n$ or $f_\vaf < 0$, the exponent vector $\vaf$ is contained in the interior of a face of $\New(f)$ that also contains $\vzero$. Then for each such $\vaf$ we can find a circuit whose inner exponent is $\vaf$ and for which $\vzero$ is one of the outer exponents. Taking $\CC$ as the set of these circuits, \eqref{eq:bound} (or equivalently, \eqref{eq:bound-B}) is clearly feasible. This is the same condition as the \emph{nondegeneracy} condition of \cite{SeidlerDeWolff2018} and \cite{Wang2019}.
\end{remark}

We will end this section with a toy example.

\begin{example}
Consider the polynomial $f$ given by \[f(z_1,z_2) = 1+z_2^2-z_1^2z_2^2+z_1^2z_2^6+z_1^6z_2^2.\]
This polynomial clearly has a SONC lower bound, since it has only one monomial that is not a monomial square, $-z_1^2 z_2^2$, and the exponent of that monomial is the inner exponent of the circuit $C_1=\{(0,0), (2,6), (6,2), (2,2)\}$, which contains $\vzero$ as an outer exponent and has signature $\vlam(C_1) = (\frac{1}{2},\frac{1}{4},\frac{1}{4})$. Thus, for a sufficiently large constant $\gamma$, we have $\gamma + z_1^2 z_2^6 + z_1^6 z_2^2 - z_1^2z_2^2 \geq 0$ for every $\vz$, and the remaining terms in $f$ are monomial squares.

Solving the primal-dual pair \eqref{eq:bound}-\eqref{eq:bound-dual} with $\CC=\{C_1\}$, we obtain the optimal value $\gamma^* = -\frac{7}{8}$, and the SONC decomposition
\[ f(z_1,z_2) - \frac{7}{8} = (z_2)^2 + \left( \frac{1}{8} + z_1^2z_2^6 + z_1^6z_2^2 - z_1^2z_2^2 \right); \]
the first term on the right-hand side is a monomial square, the second one is a member of $NC(C_1)$. The dual optimal solution (indexing the components in degree lexicographic order) is $\vy^* = (1,0,\frac{1}{4},\frac{1}{16},\frac{1}{16})$.

The constraint generation algorithm consists of solving two linear optimization problems: one to find the most promising circuit with $z_2^2$ as the inner monomial and one to find the most promising circuit with $z_1^2 z_2^2$ as the inner monomial. The remaining three monomials are vertices of the Newton polytope, and need not be considered. The first search is unsuccessful: $y^*_{(0,2)} = 0$, therefore no circuit with $(0,2)$ as an inner exponent can violate its corresponding power cone inequality \eqref{eq:constr}. The second linear optimization problem identifies the circuit $C_2 = \{(0,2), (6,2), (2,2)\}$, with signature $\vlam(C_2) = \left(\frac{2}{3}, \frac{1}{3}\right)$. The corresponding power cone constraint \eqref{eq:constr} is violated, since $y^*_{(2,2)} = \frac{1}{4} > 0 = y^*_{(0,2)}$.

Solving the primal-dual pair \eqref{eq:bound}-\eqref{eq:bound-dual} with $\CC=\{C_1,C_2\}$, the optimal value improves to $\gamma^* = -1$, and we obtain the SONC decomposition
\[  f(z_1,z_2) - 1 = z_1^2z_2^6 + \left( z_2^2+z_1^6z_2^2-z_1^2z_2^2 \right); \] 
the first term on the right-hand side is a monomial square, the second one is a member of $NC(C_2)$. The circuit $C_1$ is superfluous. The new optimal dual solution is $\vy^* = (1,0,0,0,0)$. Since every component of $\vy^*$ that corresponds to a non-vertex exponent is zero, there cannot be any circuits whose corresponding power cone inequality is violated, proving that we have found the optimal SONC bound. In this example, we also have $f(z_1,0) = 1$, proving that $1$ is the best possible global lower bound on $f$, that is, the optimal SONC bound is the global minimum.
\end{example}

\deletethis{

\section{A dual method for circuit generation}
A naive implementation of Algorithm \ref{alg:CG} may be rather inefficient, because the number of optimization variables in \eqref{eq:bound-B} grows linearly with the circuits. Suppose for a moment that we are not interested in the SONC decomposition certifying the bound computed by Algorithm \ref{alg:CG}, only the bound itself. Then the solution of the primal problem \eqref{eq:bound-B} is not required; both the bound computation in Step 3 and the search for violated inequalities in Step 4 can be carried out using only the solution $\vy^*$ and the optimal value $\gamma^*$ of the dual problem \eqref{eq:bound-dual}. In this case it is most attractive to use a dual interior-point method (alternatively, a primal interior-point method to solve the dual problem), since the problem \eqref{eq:bound-dual} is a standard form conic optimization problem with only a single linear equality constraint (independently of the number of circuits or other dimensions of the problem) and a cone constraint involving a cone for which an easily computable logarithmically homogeneous self-concordant barrier function is available: the intersection of a nonnegative orthant and a Cartesian product of power cones. (See \cite{Chares2009} for information on barrier functions for the generalized power cone.)

It further simplifies matters that this problem also has a Slater point (recall the discussion around Lemma \ref{thm:bound-strong-duality}), which can be used as an easily computable initial point after scaling to satisfy $y_\vzero=1$. (Although, in our implementation, see below, we employ an infeasible interior-point solver.)

Another way to put this is that in place of employing column generation on the primal problem, we could solve the dual problem using constraint generation, but untypically for constraint generation, it is now trivial to restart the optimization from a feasible solution after generating and adding new constraints. Since we are using an interior-point method with low worst-case iteration complexity (see below), warmstarting is not a concern, any starting point is fine.

In our implementation we use the open-source Matlab code \texttt{alfonso} \cite{PappYildiz2019, alfonso-github}, a nonsymmetric cone optimization code that can can directly solve \eqref{eq:bound-dual} in polynomial time using a predictor-corrector approach, with only the initial point and a barrier function for this cone as an input, without any model transformation. The worst-case complexity analysis of this method for a general optimization problem \cite[Thm.?]{?} yields that the number of iterations required to compute an $\epsilon$-optimal solution is $O(\nu^{1/2}\log(1/\epsilon))$, where $\nu$ is the barrier parameter of the barrier function used, while the number of arithmetic operations per iteration is $O(m^3 + B)$ where $m$ is the number of optimization variables and $B$ is the number of floating point operations required to compute the Hessian of the barrier function at a given point. (In each iteration of the interior-point method, we need compute the Hessian of the barrier function at the current point twice, once for the predictor step and once for the corrector, and then compute an appropriate Newton step in $O(m^3)$ flops.) The space complexity (required memory) is $O(m^2)$, the bottleneck being the storage of the barrier Hessian.

In our case, $m=|\supp(f)|$, hence the required $O(m^2)$ memory remains constant as the number of circuits increases. The computation of the Hessian of the barrier requires $O(|\CC|n^2 + m^2)$ flops, and the Newton step requires $O(m^3)$ flops. The barrier parameter $\nu$ is of order $O(|\CC|\cdot n)$. Altogether, this results in a total time complexity of $O(|\CC|^{1/2}\cdot n^{1/2} \cdot \max(|\CC|\cdot n^2, |\supp(f)|^3))$. For a given sparse polynomial $f$, $\supp(f)$ is a manageable constant throughout the circuit generation procedure, $n$ is (without loss of generality) smaller than $m$, and the initial $|\CC|$ is also not greater than $m$. Therefore, in the early iterations, $|\CC|n^2 \leq |\supp(f)|^3$. In all of our experiments, $|\CC|$ remains fairly small, in agreement with the fact that SONC polynomials have SONC decompositions with no more than $|\supp(f)|$ circuits. Thus, as the circuit generation progresses and $|\CC|$ grows, the solution time for \eqref{eq:bound-dual} grows rather amicably, proportional to $|\CC|^{1/2}$. In our experiments (Section \ref{sec:numericalexperiments}), the only parameter limiting the computation of SONC bounds was $|\supp(f)|$, which could be increased as long as the solution of systems of linear equations of order $\supp(f)$ is manageable.

\subsection{Reconstructing the optimal SONC decomposition} In the previous analysis we assumed that optimal solution $\vx$ of \eqref{eq:bound-B} is not required, but in many applications it is important to compute the SONC decomposition certifying the computed SONC bounds. A simple but perhaps inelegant solution is to apply the above dual method first to compute the circuits $\CC$ of the decomposition, and only then solve \eqref{eq:bound-B} with this set $\CC$. This allows the computation of an optimal decomposition at the price of solving only one large instance of the problem \eqref{eq:bound-B}. We show that we can compute the optimal decomposition much faster, without any optimization.
.....
}

\section{Implementation} \label{sec:implementation}
In our implementation we use the open-source Matlab code \texttt{alfonso} \cite{PappYildiz2019, alfonso-github}, a nonsymmetric cone optimization code that can directly solve the primal-dual pair (\ref{eq:bound}-\ref{eq:bound-dual}) using a predictor-corrector approach without any model transformation. (In particular, there is no need to represent the SONC cone or its dual as an affine slice of a Cartesian product of exponential, relative entropy, or semidefinite cones.) \texttt{alfonso} requires only an interior point in the primal cone and a logarithmically homogeneous self-concordant barrier function for the primal cone as input. Since the primal cone is a Cartesian product of nonnegative half-lines and dual cones of generalized power cones, both an easily computable initial point and a suitable barrier function are readily available; see, for example, \cite{Chares2009}.

Alternatively, we can use \eqref{eq:bound-dual} as the ``primal'' problem for \texttt{alfonso}. This cone is an intersection of generalized power cones and a nonnegative orthant, so the barrier function is once again readily available, this time as the sum of well-known barrier functions. Furthermore, the Slater point for this problem (recall the discussion around Lemma~\ref{thm:bound-strong-duality}) can be used as an easily computable initial point after scaling to satisfy the only non-homogeneous constraint $y_\vzero=1$. In our implementation we used the latter variant. When started with a feasible initial solution, \texttt{alfonso} maintains feasibility throughout. Therefore, using the dual variant and the dual Slater point as an initial feasible solution, we are guaranteed that are our near-optimal solution to (\ref{eq:bound}-\ref{eq:bound-dual}) is dual feasible, and thus the dual optimal value is a lower bound on the minimum even if the other optimality conditions are not satisfied to a high tolerance.

In our first set of experiments (smaller instances with general Newton polytopes) we used the two-phase version of the circuit generation algorithm. In our second set of experiments (larger problems with simplex Newton polytopes) it was easy to find an initial set of circuits, and started with Phase II. In the circuit generation steps, we added every promising circuit identified (up to one circuit for each monomial that is not a vertex of the Newton polytope).

The linear optimization problems used in circuit generation were solved using Matlab's built-in \texttt{linprog} function with options that ensure that an optimal basic feasible solution is returned (and not the analytic center of the optimal face).

 %The worst-case complexity analysis of this method for a general optimization problem \cite[Thm.?]{?} yields that the number of iterations required to compute an $\epsilon$-optimal solution is $O(\nu^{1/2}\log(1/\epsilon))$, where $\nu$ is the barrier parameter of the barrier function used, while the number of arithmetic operations per iteration is $O(m^3 + B)$ where $m$ is the number of optimization variables and $B$ is the number of floating point operations required to compute the Hessian of the barrier function at a given point. (In each iteration of the interior-point method, we need compute the Hessian of the barrier function at the current point twice, once for the predictor step and once for the corrector, and then compute an appropriate Newton step in $O(m^3)$ flops.) The space complexity (required memory) is $O(m^2)$, the bottleneck being the storage of the barrier Hessian.

\section{Numerical experiments} \label{sec:numericalexperiments}

\subsection{The Seidler--de Wolff benchmark problems}

The first set of instances the algorithm was tested on were problems from the database of unconstrained minimization benchmark problems accompanying the paper \cite{SeidlerDeWolff2018}. Each instance is a polynomial generated randomly in a way that the polynomial is guaranteed to have a lower bound and a prescribed number of unknowns, degree, and cardinality of support (number of monomials with nonzero coefficients). Since this database is enormous (it has over $30\,000$ instances), we opted to use only the largest and most difficult instances: the ones with general (not simplex) Newton polytopes and $500$ monomials in their support. There are 438 such instances; the number of unknowns $n$ in these instances ranges from 4 to 40, the degree $d$ between 6 and 60. These are indeed very sparse polynomials, the dimensions $\binom{n+d}{d}$ of their corresponding spaces of ``dense'' polynomials ranges from 8008 to $6\cdot 10^{25}$.

All experiments were run using Matlab 2017b on a Dell Optiplex 7050 desktop with a 3.6GHz Intel Core i7 CPU and 32GB RAM.

Figure \ref{fig:polydb-iterations} shows the histogram of the total number of circuit generation iterations in Phase I and Phase II combined. The smallest possible value is therefore $2$ (in the case when the initial set of circuits is optimal). The histogram shows that in the vast majority of these instances no more than 1 additional iteration was needed, that is, all necessary circuits were either among the initial ones, or were identified in the first circuit generation step of Phase II. Correspondingly, the scatterplot in Figure \ref{fig:polydb-circuits-time} shows that most of the instances could be solved under a minute, and that the total number of circuits needed to certify the optimal bound was under 1000. (Recall that the initial set of circuits is below $|\supp(f)|=500$.) It is perhaps interesting to note that even in the ``hardest'' instance, the algorithm generated fewer than 4500 circuits before the optimal bound was found. This was the only instance where the total running time exceeded one hour; most instances were solved under one minute, and nearly all of them under 5 minutes. There was no discernible pattern indicating what made the difficult instances difficult. In particular, the number of unknowns and the degree alone are not good predictors of the number of circuits or the number of circuit generation iterations.

The optimal solutions or the best known lower bounds are not available in the database. However, upper bounds on the minima of the polynomials can be computed using multi-start local optimization. For simplicity and reproducibility, we used the \texttt{NMinimize} function in Mathematica (version 11.3) with default settings to compute approximate minimizers for each of the 438 instances. As the histogram of optimality gaps in Figure \ref{fig:polydb-gap} shows, the computed SONC bounds were near-optimal for each instance. This is somewhat surprising, and merits further investigation, as it is in general not guaranteed that a polynomial that is bounded from below has a SONC bound at all; one certainly cannot expect that this bound will always be close to (or equal to) the infimum of the polynomial. Similarly, it cannot be hoped that the local minimum returned by Mathematica is a global minimum. Nevertheless, in each of these instances, the SONC bound was within 1.2\% of the global minimum of the polynomial, and with the exception of 46 instances (=10.5\%), the relative optimality gap was within $10^{-6}$.

\begin{figure}
	\centering
	\includegraphics[width=0.95\textwidth]{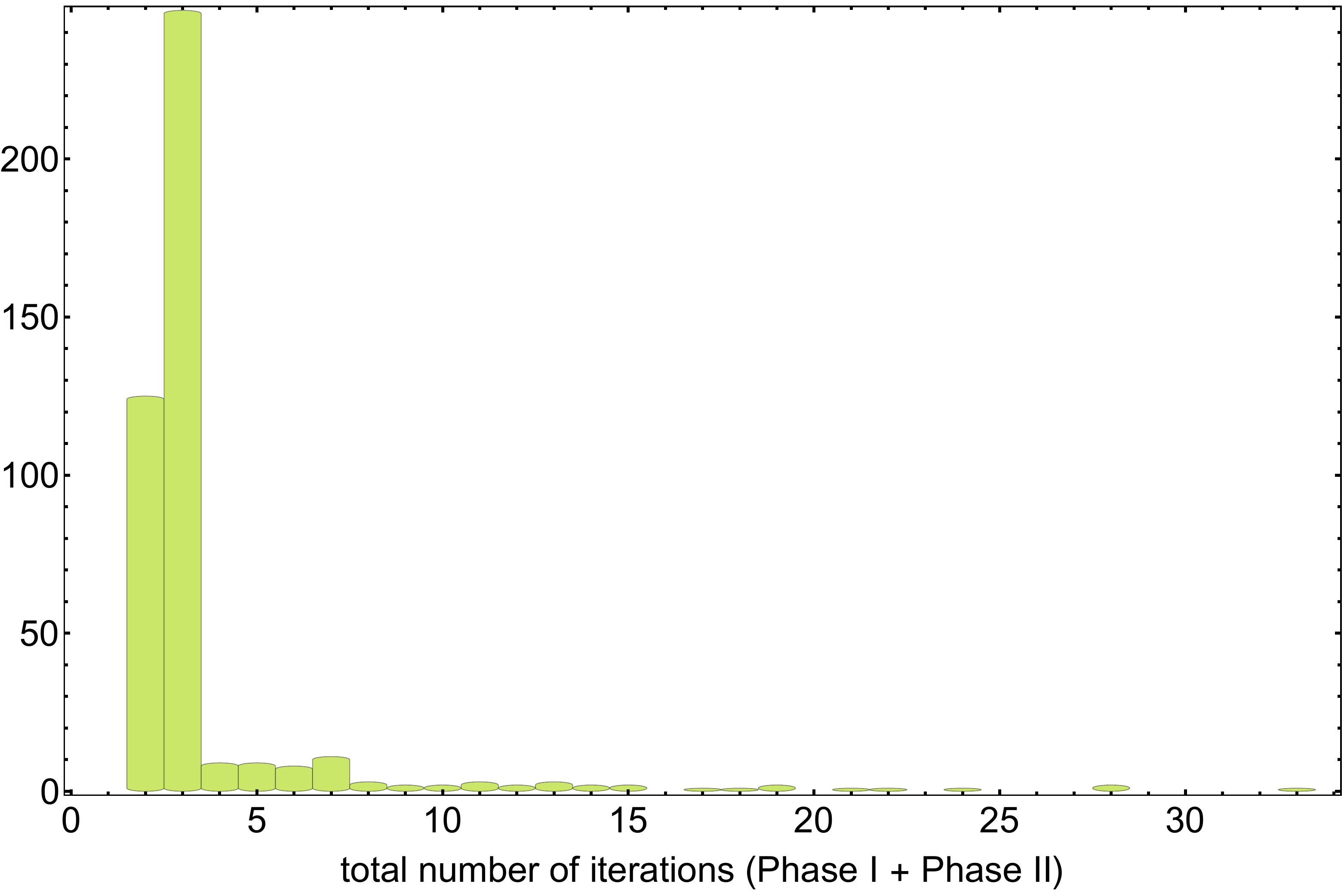}\caption{Histogram of the total number of circuit generation iterations for the largest instances of the Seidler--de Wolff instances. (438 instances; each with 500 monomials, with a varying number of unknowns and degree.) The smallest possible value is $2$ (one Phase I iteration and one Phase II iteration). Most instances were solved in two or three iterations.}\label{fig:polydb-iterations}
\end{figure}

\begin{figure}
	\centering
	\includegraphics[width=0.95\textwidth]{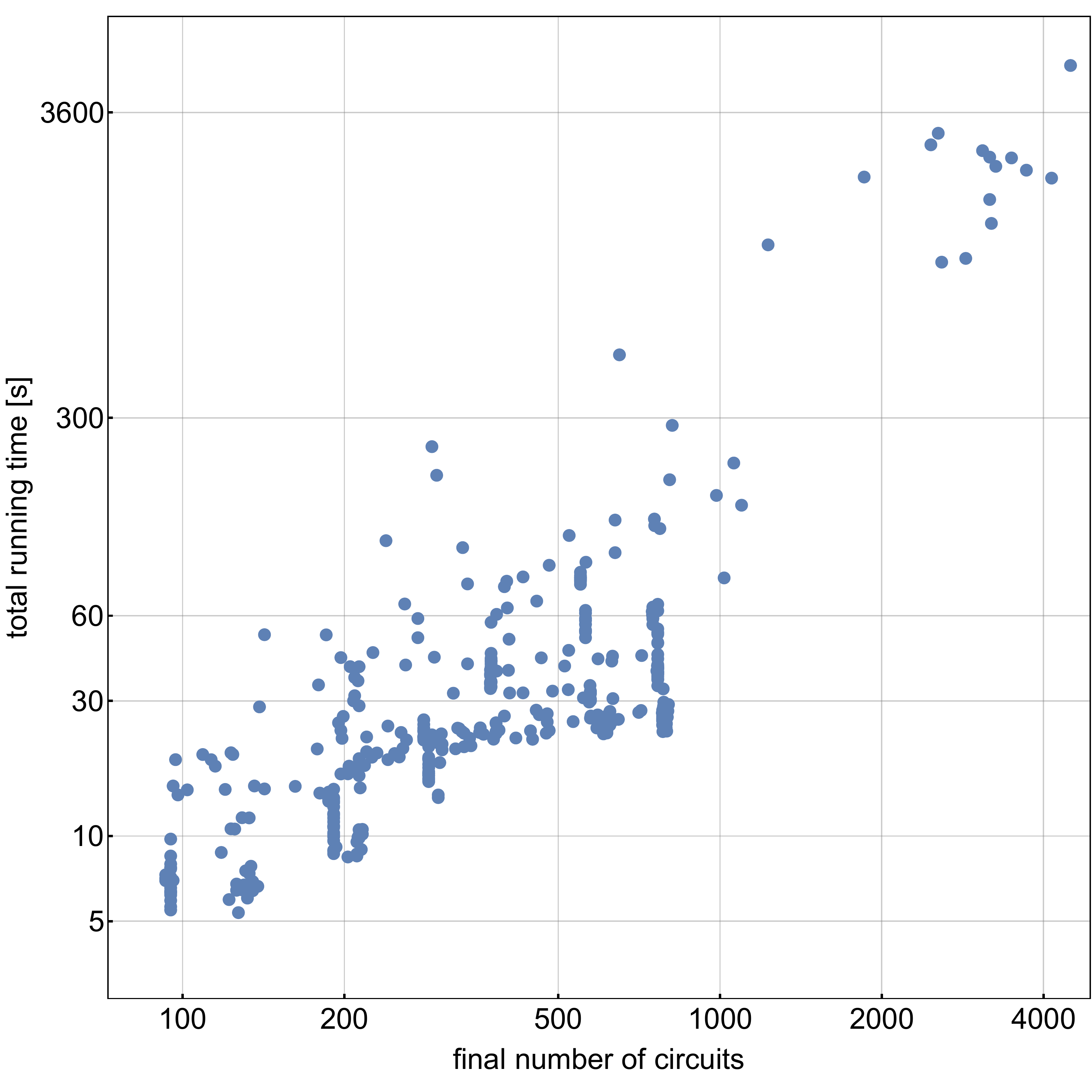}\caption{Scatter plot of the number of circuits in the final iteration of the algorithm and the total running time of the algorithm for the Seidler--de Wolff instances, shown on a logarithmic scale for better visibility. Each dot represents an instance. Since the number of iterations was uniformly small for most instances, the running times and the final number of circuits correlate well. Most instances were solved under a minute, and nearly all of them under 5 minutes. One instance took over an hour to solve.}\label{fig:polydb-circuits-time}
\end{figure}

\begin{figure}
	\centering
	\includegraphics[width=0.95\textwidth]{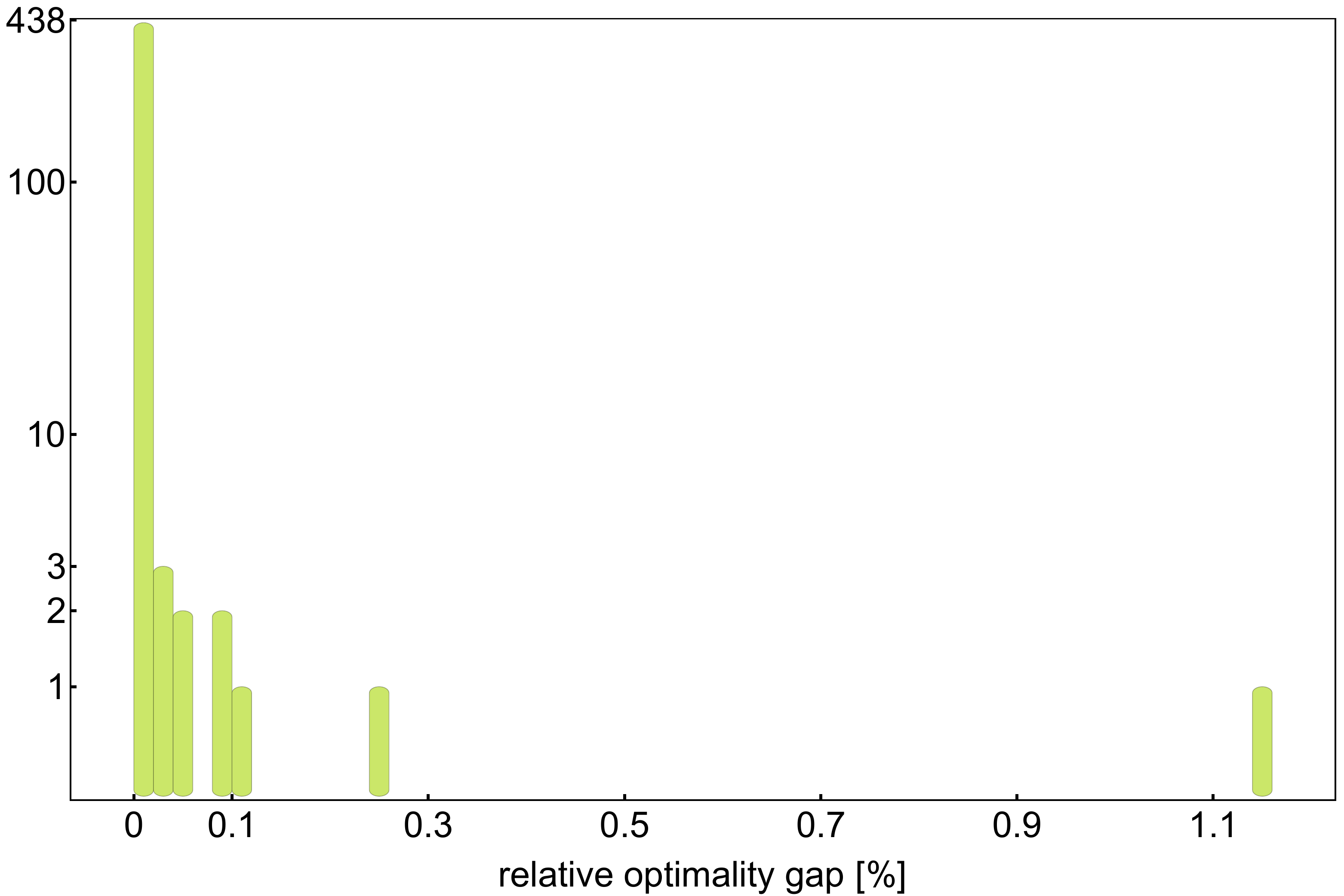}\caption{Histogram of the relative optimality gaps obtained for the Seider--de Wolff instances. Surprisingly, the computed SONC lower bounds were close to the optimal value for each instance. (Note the logarithmic scale on the vertical axis.) The majority of the instances had an optimality gap of $10^{-6}$ or smaller; too small for the resolution of this picture.}\label{fig:polydb-gap}
\end{figure}

\subsection{Larger instances}

The second set of instances were generated in a somewhat similar fashion as those in the previous set, but the parameters were increased to test the limits of our approach (in particularly, increasing the size of the support above 500). The instances for this experiment were polynomials of degree $d=8$ with $n=25$ unknowns. The random supports and coefficients were generated in the following manner: the constant monomial and the monomials $x_i^d$ were given random integer coefficients between $1$ and $5$, then a random subset of monomials with componentwise even exponents with total degree less than $d$ were selected (without replacement) and given a random non-zero integer coefficient between $-5$ and $5$. The size of the support was varied in $5\%$ increments up to the maximum of $3301$ (the number of componentwise even $25$-variate monomials with total degree less than $d=8$).

Generating the instances in this fashion achieves the following: (1) it is clear a priori that the polynomials can be bounded from below; (2) the Newton polytope $\New(f)$ is known in advance (an $(n+1)$-simplex whose vertices correspond to the monomials $0$ and $x_1^d,\dots,x_n^d$); (3) Phase I can be skipped, and Phase II can be started with an easily computable set of circuits: every exponent in $\supp(f)\setminus V$ is the inner exponent of exactly one initial circuit whose outer exponents are appropriate vertices of the simplex Newton polytope.

Componentwise even monomials were chosen to maximize the number of circuits that can be formed by points in the support and thus make the problems more challenging. (Every exponent of the support other than the vertices of the Newton polytope can be an inner or outer monomial of a number of circuits.) One can also think of the lower bounding of these polynomials over $\R^n$ as problems of bounding polynomials $f$ of degree $4$ over the nonnegative orthant by first applying the change of variables $z_i \leftarrow w_i^2$ and then bounding the polynomial $\vw \to f(\vw^2)$ over $\R^n$.

Each experiment was replicated 10 times (that is, 10 randomly generated instances were solved for each problem size) using the same software and hardware as in the first set of experiments. Figure \ref{fig:exp2-times} shows the distribution of running times for each problem size. The running time increases fairly moderately (approximately cubically) as the number of monomials increases; it remained under 1.5 hours for every instance. To see where the increase in running time comes from, in Figure \ref{fig:exp2-factor} we plot the ratio between the number of circuits at the end of the circuit generation algorithm and the number of initial circuits, and in Figure \ref{fig:exp2-iterations} we plot the number of circuit generation iterations. The ratio appears to increase only linearly with the initial number of monomials, showing that the circuit generation algorithm is very effective in choosing the right circuits to add to the formulation out of the exponentially many circuits. (We have no theoretical explanation for this). Although the number of circuit generation iterations increases with increasing problem sizes (as expected), this increase is very slow (clearly sublinear); most instances were solved in fewer than 8 iterations. Figure \ref{fig:exp2-trajectories} shows the evolution of the number of circuits for each instance as the algorithm progresses.

\begin{figure}
	\centering
	\includegraphics[width=0.95\textwidth]{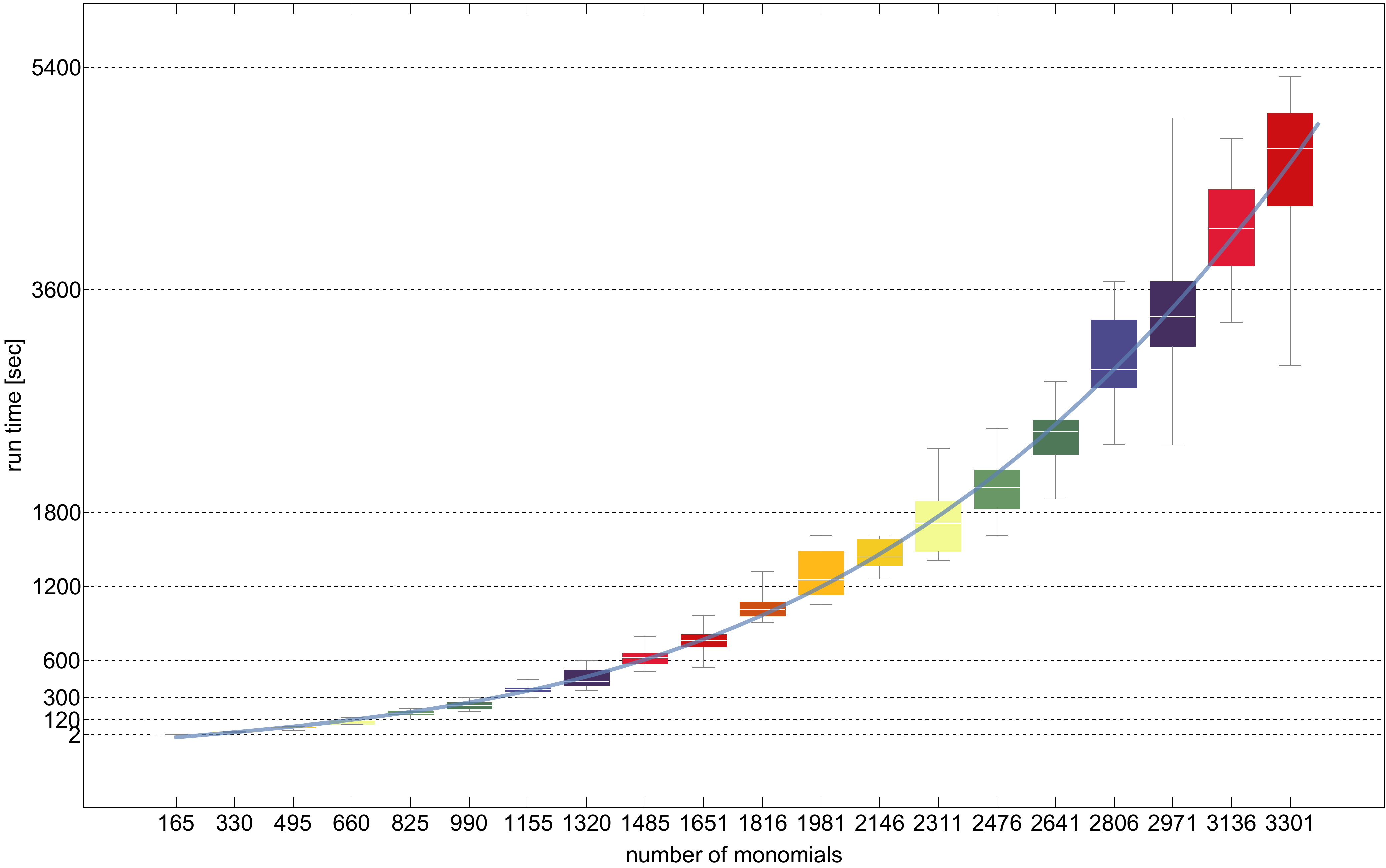}\caption{Box-whisker plot of total running times as a function of problem size from the second experiment. Problem size (horizontal axis) is measured by the number of monomials. Each box represents results from 10 experiments with random polynomials of the same size. A cubic function fitted to the mean values is also shown.}\label{fig:exp2-times}
\end{figure}

\begin{figure}
	\centering
	\includegraphics[width=0.95\textwidth]{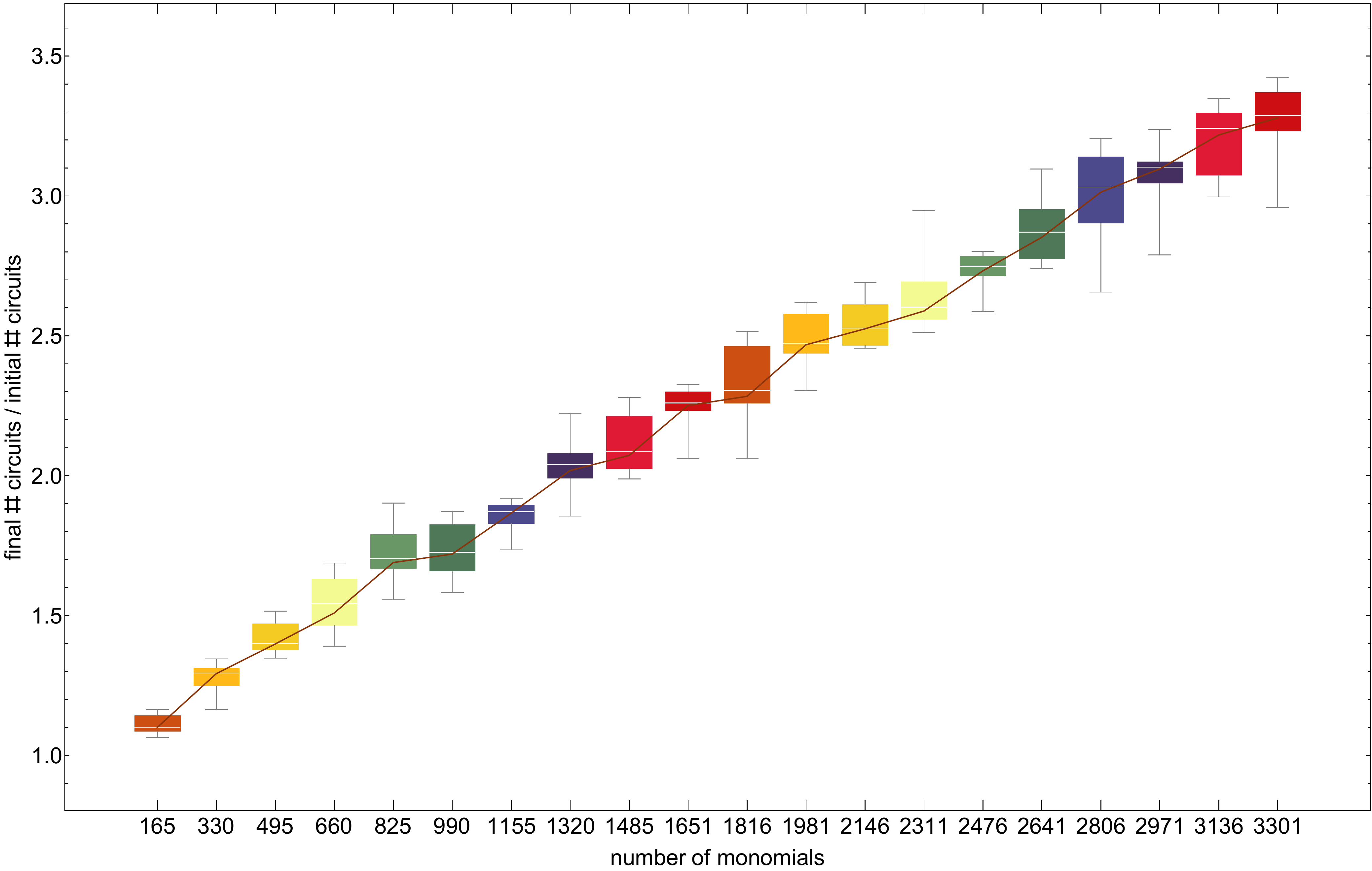}\caption{Box-whisker plot of the ratio between the final number of circuits and the initial number of circuits as a function of problem size from the second experiment. Problem size (horizontal axis) is measured by the number of monomials. Each box represents results from 10 experiments with random polynomials of the same size. The ratio appears to increase only linearly.}\label{fig:exp2-factor}
\end{figure}

\begin{figure}
	\centering
	\includegraphics[width=0.95\textwidth]{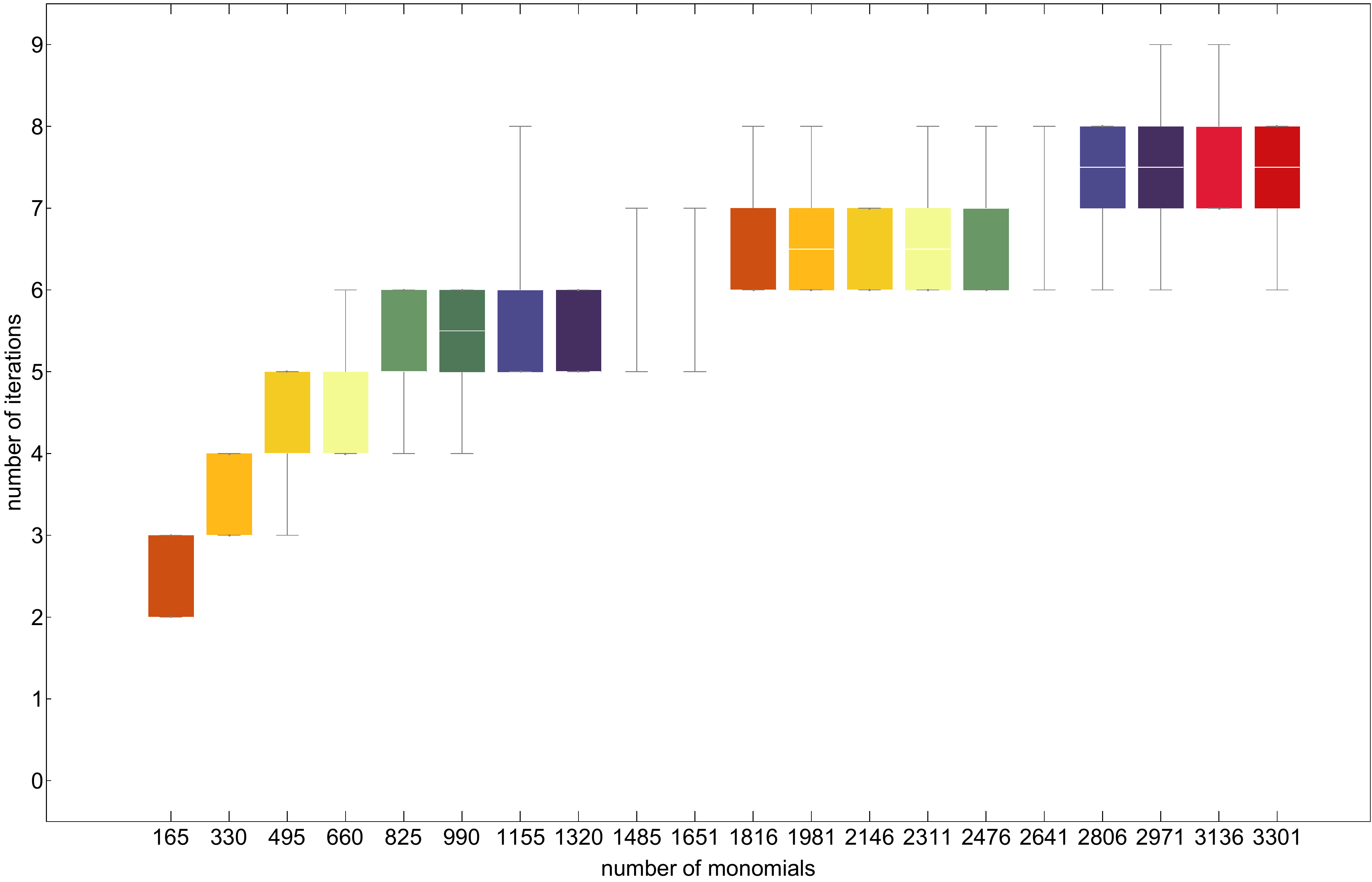}\caption{Box-whisker plot of the number of circuit generation iterations as a function of problem size from the second experiment. Problem size (horizontal axis) is measured by the number of monomials. Each box represents results from 10 experiments with random polynomials of the same size. The ratio appears to increase very slowly (sublinearly).}\label{fig:exp2-iterations}
\end{figure}

\begin{figure}
	\centering
	\includegraphics[width=0.95\textwidth]{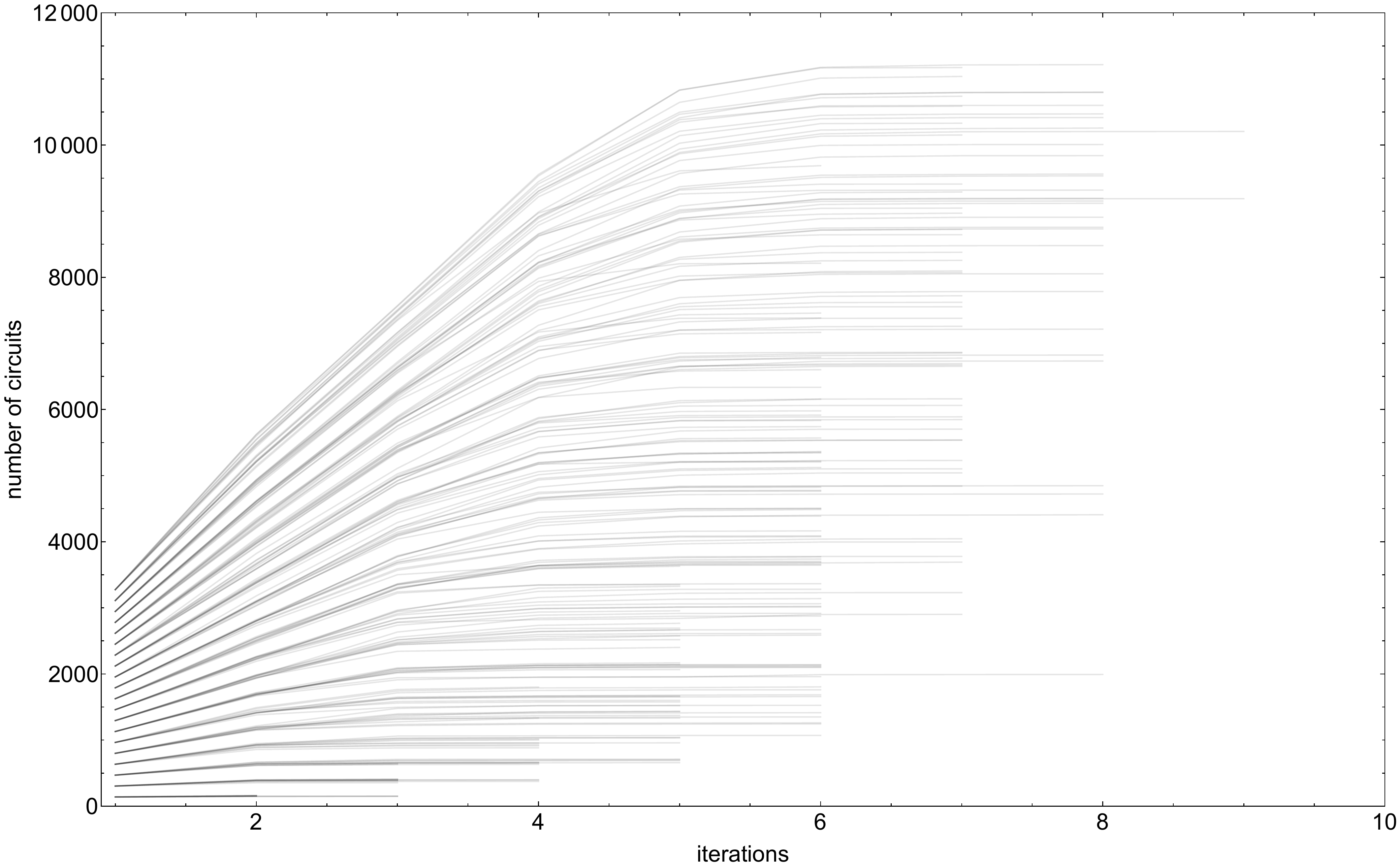}\caption{Diagram showing the number of circuits in each iteration for each instance of the second experiment. Most circuits are added in the first few iterations of the algorithm, in which a new circuit is added for nearly each monomial; later iterations add circuits more selectively. For most instances, several iterations add only a very small number of circuits. The objective function values (not shown) also reveal that in these iterations the bound often does not improve, but the promising circuits need to be added in order to certify the optimality of the bound.}\label{fig:exp2-trajectories}
\end{figure}

\section{Discussion} \label{sec:discussion}

The computational results confirm that the proposed approach is well-suited for bounding sparse polynomials even when the number of unknowns and the degree are fairly large. Theoretically, the primary driver of the running time is the size of the support, which determines the number of circuits required for an optimal SONC decomposition. The number of circuit generation iterations also appears to depend on the support size, but this dependence was surprisingly mild in all the experiments. (This does not have an apparent theoretical support, but is in line with our experience with column generation approaches in other settings.) Additionally, the dimension of the power cones (and dual power cones) may depend on the number of unknowns, since each circuit may have up $n+1$ outer exponents for polynomials with $n$ unknowns. However, assuming that the support size and the number of unknowns are fixed, the degree of the polynomials does not have an additional impact on the time complexity of the algorithm.

The second phase of the circuit generation approach finds the optimal SONC bound (and the corresponding circuits and SONC decomposition) once a SONC bound is known to exist from Phase I. The first phase, however, does something slightly weaker than certifying the existence or non-existence of a SONC bound: it finds circuits to prove a target lower bound if possible; in other words, for a given polynomial $f$ and constant $c$, it can decide whether $f+c$ is SONC or not. If it is, it finds a SONC decomposition of $f+c$, if it is not, it finds a (numerical) certificate of $f+c$ being outside of SONC. This theoretical gap cannot be closed with a numerical method: we cannot certify the non-existence of SONC bounds in general, since the set of polynomials with a finite SONC lower bound is not closed. For instance, $f_\varepsilon(\vz) \defeq (1+\varepsilon)z_1^2 - 2z_1z_2 + z_2^2 - 2z_1$ has a SONC lower bound for every $\varepsilon>0$ (because $f_\varepsilon+1/\varepsilon$ is SONC) but $f_0$ does not have a SONC lower bound (because it is not bounded from below). Practically, this means that we can run the first phase with a ``large'' value of $c$ and either conclude that a ``useful'' SONC bound not does not exist (because $f+c$ is not SONC) or that $f$ has a SONC lower bound (greater than $-c$); in the latter case Phase II can compute the optimal SONC lower bound.

There are many possible extensions of the algorithm proposed in this paper. The theoretically most straightforward one is to apply the same principle to general  optimization problems in which the nonnegativity of an unknown polynomial appears as a constraint. Replacing the nonnegativity constraint with a SONC constraint, this leads to optimization problems similar to the ones we considered, except that every coefficient of the polynomials in question becomes a decision variable (rather than only the constant term being an optimization variable), and the problem may have additional optimization variables. A circuit generation procedure can be derived entirely analogously for problems of this type as long as the additional optimization variables are related to the coefficients of the SONC polynomials through linear constraints.

One may also use this approach to generate circuits for an optimal decomposition of a polynomial into the sum of a SONC polynomial and a sum-of-squares (SOS) polynomial. Theoretically, neither the SOS nor the SONC bound is always better than the other (bivariate counterexamples are easy to find); a combined SOS+SONC bound would of course be at least as good as either of them. This is not a straightforward computational problem, however, because SOS bounds are typically computed using semidefinite programming algorithms, using software that cannot handle the power cone constraints used in our algorithm. However, the primal-dual algorithm and software used in this paper (\texttt{alfonso}) was also used earlier to efficiently compute SOS bounds for polynomials \cite{PappYildiz2019}, implying that the same code could also be used to compute SOS+SONC bounds. The most recent version (version 9) of the commercial conic optimization software Mosek \cite{mosek9} also supports the simultaneous use of semidefinite and power cone constraints.

Should the number of circuits generated by the algorithm become prohibitively large, one may consider an improved version of Algorithm~\ref{alg:CG} which does not only add new promising circuits but also attempts to remove the unnecessary ones in each iteration. This problem did not arise in our experiments (the number of circuits never increased above 10 times the number of circuits used in the optimal SONC decomposition), hence we did not pursue this direction in the paper. We note however that dropping all circuits not used in the last iteration may lead to cycling (the same circuits being added again in the next iteration and than dropped again). An example of a constraint generation algorithm for convex optimization that drops unnecessary cone constraints but safeguards against cycling and could likely be adapted to our problem is \cite{MehrotraPapp2014}.

Lastly, we leave it for future work to implement an extension of the proposed method to a hybrid symbolic-numerical method that generates rigorous global lower bounds and certificates that can be verified in exact arithmetic from the numerical SONC decompositions computed by our algorithm. Since the numerical method used in our implementation is a primal-dual interior-point approach that computes a \emph{strictly interior} feasible solution $\vy^*$ to the problem \eqref{eq:bound-dual}, it is a trivial matter to compute a nearby rational feasible solution $\vy_{\text{rat}}$ to the same problem by componentwise rounding the numerical vector $\vy^*$ to a close enough rational vector without violating any of the cone constraints. Finally, the problem's only equality constraint can be satisfied exactly by scaling $\vy_{\text{rat}}$ (although this does leave a square root in the final symbolic solution). The resulting dual objective function value $-\vf^\T\vy_{\text{rat}}$ is a rigorous global lower bound on $f$, close to the numerically obtained bound, whose correctness can be verified in exact arithmetic by verifying the strict feasibility of $\vy_\text{rat}$. The reconstruction of a primal certificate, that is, a verifiable exact SONC decomposition by computing a rational feasible solution $(p_1,\dots,p_N)$ of the primal problem \eqref{eq:bound-B} from the near-optimal, and only near-feasible, numerical solution is a more complicated matter.

\deletethis{
???
The model in the proof of Lemma \ref{thm:bound-attainment} has a Slater point under all sorts of conditions (though not WLOG). Could its dual \eqref{eq:bound-dual} always have attainment???
}

\section*{Acknowledgments}
The author is grateful to Mareike Dressler (UCSD) for pointing out the reference to Jie Wang's recent work \cite{Wang2019} on the support of SONC polynomials.

\bibliographystyle{siamplain}
\bibliography{circuit_generation}

\begin{thebibliography}{10}

\bibitem{AhmadiMajumdar2014}
{\sc A.~A. Ahmadi and A.~Majumdar}, {\em {DSOS} and {SDSOS} optimization: {LP}
  and {SOCP}-based alternatives to sum of squares optimization}, in 48th Annual
  Conference on Information Sciences and Systems (CISS), IEEE, 2014, pp.~1--5,
  \url{https://doi.org/10.1109/CISS.2014.6814141}.

\bibitem{AhmadiMajumdar2016}
{\sc A.~A. Ahmadi and A.~Majumdar}, {\em Some applications of polynomial
  optimization in operations research and real-time decision making},
  Optimization Letters, 10 (2016), pp.~709--729,
  \url{https://doi.org/10.1007/s11590-015-0894-3}.

\bibitem{AylwardItaniParrilo2007}
{\sc E.~M. Aylward, S.~M. Itani, and P.~A. Parrilo}, {\em Explicit {SOS}
  decompositions of univariate polynomial matrices and the
  {K}alman-{Y}akubovich-{P}opov lemma}, in 46th IEEE Conference on Decision and
  Control, Dec 2007, pp.~5660--5665,
  \url{https://doi.org/10.1109/CDC.2007.4435026}.

\bibitem{BachocVallentin2008}
{\sc C.~Bachoc and F.~Vallentin}, {\em New upper bounds for kissing numbers
  from semidefinite programming}, Journal of the American Mathematical Society,
  21 (2008), pp.~909--924, \url{https://doi.org/10.1090/S0894-0347-07-00589-9}.

\bibitem{BallingerBlekhermanCohnGiansiracusaKellySchurmann2009}
{\sc B.~Ballinger, G.~Blekherman, H.~Cohn, N.~Giansiracusa, E.~Kelly, and
  A.~Sch\"urmann}, {\em Experimental study of energy-minimizing point
  configurations on spheres}, Experimental Mathematics, 18 (2009),
  pp.~257--283, \url{https://doi.org/10.1080/10586458.2009.10129052}.

\bibitem{BarakEtal2019}
{\sc B.~Barak, S.~Hopkins, J.~Kelner, P.~K. Kothari, A.~Moitra, and
  A.~Potechin}, {\em A nearly tight sum-of-squares lower bound for the planted
  clique problem}, SIAM Journal on Computing, 48 (2019), pp.~687--735,
  \url{https://doi.org/10.1137/17M1138236}.

\bibitem{BlekhermanParriloThomas2013}
{\sc G.~Blekherman, P.~A. Parrilo, and R.~R. Thomas}, eds., {\em Semidefinite
  optimization and convex algebraic geometry}, vol.~13 of MOS-SIAM Series on
  Optimization, Society for Industrial and Applied Mathematics (SIAM),
  Philadelphia, PA, 2013.

\bibitem{ChandrasekaranShah2016}
{\sc V.~Chandrasekaran and P.~Shah}, {\em Relative entropy relaxations for
  signomial optimization}, SIAM Journal on Optimization, 26 (2016),
  pp.~1147--1173, \url{https://doi.org/10.1137/140988978}.

\bibitem{Chares2009}
{\sc R.~Chares}, {\em Cones and interior-point algorithms for structured convex
  optimization involving powers and exponentials}, PhD thesis, Universit{\'e}
  Catholique de Louvain, 2009.

\bibitem{DeitsTedrake2015}
{\sc R.~Deits and R.~Tedrake}, {\em Efficient mixed-integer planning for {UAVs}
  in cluttered environments}, in {Proceedings of the 2015 IEEE International
  Conference on Robotics and Automation (ICRA)}, May 2015, pp.~42--49,
  \url{https://doi.org/10.1109/ICRA.2015.7138978}.

\bibitem{DickinsonGijben2014}
{\sc P.~J.~C. Dickinson and L.~Gijben}, {\em On the computational complexity of
  membership problems for the completely positive cone and its dual},
  Computational Optimization and its Applications, 57 (2014), pp.~403--415,
  \url{https://doi.org/10.1007/s10589-013-9594-z}.

\bibitem{DresslerIlimanDeWolff2017}
{\sc M.~Dressler, S.~Iliman, and T.~de~Wolff}, {\em A {P}ositivstellensatz for
  sums of nonnegative circuit polynomials}, SIAM Journal on Applied Algebra and
  Geometry, 1 (2017), pp.~536--555, \url{https://doi.org/10.1137/16M1086303}.

\bibitem{GhaddarMarecekMevissen2016}
{\sc B.~Ghaddar, J.~Marecek, and M.~Mevissen}, {\em Optimal power flow as a
  polynomial optimization problem}, IEEE Transactions on Power Systems, 31
  (2016), pp.~539--546, \url{https://doi.org/10.1109/TPWRS.2015.2390037}.

\bibitem{GhasemiMarshall2012}
{\sc M.~Ghasemi and M.~Marshall}, {\em Lower bounds for polynomials using
  geometric programming}, SIAM Journal on Optimization, 22 (2012),
  pp.~460--473, \url{https://doi.org/10.1137/110836869}.

\bibitem{GhasemiMarshall2013}
{\sc M.~Ghasemi and M.~Marshall}, {\em Lower bounds for a polynomial on a basic
  closed semialgebraic set using geometric programming}, arXiv preprint
  1311.3726,  (2013).

\bibitem{GoluskinFantuzzi2019}
{\sc D.~Goluskin and G.~Fantuzzi}, {\em Bounds on mean energy in the
  {K}uramoto-{S}ivashinsky equation computed using semidefinite programming},
  Nonlinearity, 32 (2019), p.~1705,
  \url{https://doi.org/10.1088/1361-6544/ab018b}.

\bibitem{Harrison2007}
{\sc J.~Harrison}, {\em Verifying nonlinear real formulas via sums of squares},
  in Theorem Proving in Higher Order Logics, K.~Schneider and J.~Brandt, eds.,
  Berlin, Heidelberg, 2007, Springer Berlin Heidelberg, pp.~102--118.

\bibitem{HenrionGarulli2005}
{\sc D.~Henrion and A.~Garulli}, eds., {\em Positive polynomials in control},
  vol.~312 of Lecture Notes in Control and Information Sciences,
  Springer-Verlag, Berlin, 2005, \url{https://doi.org/10.1007/b96977}.

\bibitem{gloptipoly}
{\sc D.~Henrion and J.-B. Lasserre}, {\em {GloptiPoly}: Global optimization
  over polynomials with {M}atlab and {SeDuMi}}, {ACM} Transactions on
  Mathematical Software, 29 (2003), pp.~165--194,
  \url{https://doi.org/10.1145/779359.779363}.

\bibitem{IlimanDeWolff2016}
{\sc S.~Iliman and T.~de~Wolff}, {\em Amoebas, nonnegative polynomials and sums
  of squares supported on circuits}, Research in the Mathematical Sciences, 3
  (2016), p.~9, \url{https://doi.org/10.1186/s40687-016-0052-2}.

\bibitem{JoszMaeghtPanciaticiGilbert2015}
{\sc C.~Josz, J.~Maeght, P.~Panciatici, and J.~C. Gilbert}, {\em Application of
  the moment-{SOS} approach to global optimization of the {OPF} problem}, IEEE
  Transactions on Power Systems, 30 (2015), pp.~463--470,
  \url{https://doi.org/10.1109/TPWRS.2014.2320819}.

\bibitem{KuangEtal2017}
{\sc X.~Kuang, B.~Ghaddar, J.~Naoum-Sawaya, and L.~F. Zuluaga}, {\em
  Alternative {LP} and {SOCP} hierarchies for {ACOPF} problems}, IEEE
  Transactions on Power Systems, 32 (2017), pp.~2828--2836,
  \url{https://doi.org/10.1109/TPWRS.2016.2615688}.

\bibitem{KuntzThomasStanBarahona2019}
{\sc J.~Kuntz, P.~Thomas, G.-B. Stan, and M.~Barahona}, {\em Bounding the
  stationary distributions of the chemical master equation via mathematical
  programming}, Journal of Chemical Physics, 151 (2019), p.~034109,
  \url{https://doi.org/10.1063/1.5100670}.

\bibitem{Lasserre2001}
{\sc J.~B. Lasserre}, {\em Global optimization with polynomials and the problem
  of moments}, SIAM Journal on Optimization, 11 (2001), pp.~796--817,
  \url{https://doi.org/10.1137/S1052623400366802}.

\bibitem{MehrotraPapp2014}
{\sc S.~Mehrotra and D.~Papp}, {\em A cutting surface algorithm for
  semi-infinite convex programming with an application to moment robust
  optimization}, SIAM Journal on Optimizaton, 24 (2014), pp.~1670--1697.
\newblock \url{http://dx.doi.org/10.1137/130925013}.

\bibitem{mosek9}
{\sc {MOSEK ApS}}, {\em {MOSEK} {O}ptimization {S}uite Release {9.1.5}}, 2019,
  \url{https://docs.mosek.com/9.1/intro.pdf}.

\bibitem{Nesterov2000}
{\sc Y.~Nesterov}, {\em Squared functional systems and optimization problems},
  in High performance optimization, H.~Frenk, K.~Roos, T.~Terlaky, and
  S.~Zhang, eds., vol.~33 of Applied Optimization, Kluwer Academic Publishers,
  Dordrecht, 2000, pp.~405--440,
  \url{https://doi.org/10.1007/978-1-4757-3216-0_17}.

\bibitem{Papp2012}
{\sc D.~Papp}, {\em Optimal designs for rational function regression}, Journal
  of the American Statistical Association, 107 (2012), pp.~400--411,
  \url{https://doi.org/10.1080/01621459.2012.656035},
  \url{http://dx.doi.org/10.1080/01621459.2012.656035}.

\bibitem{PappYildiz2019}
{\sc D.~Papp and S.~Y{\i}ld{\i}z}, {\em \bibhack{1}{Sum-of-squares optimization
  without semidefinite programming}}, SIAM Journal on Optimization, 29 (2019),
  pp.~822--851, \url{https://doi.org/10.1137/17M1160124}.

\bibitem{alfonso-github}
{\sc D.~Papp and S.~Y{\i}ld{\i}z}, {\em \bibhack{2}{alfonso: ALgorithm FOr
  Non-Symmetric Optimization}}.
\newblock \url{https://github.com/dpapp-github/alfonso}, 2019.

\bibitem{ParriloThesis2000}
{\sc P.~A. Parrilo}, {\em Structured Semidefinite Programs and Semialgebraic
  Geometry Methods in Robustness and Optimization}, PhD thesis, California
  Institute of Technology, May 2000.

\bibitem{sostools}
{\sc S.~Prajna, A.~Papachristodoulou, P.~Seiler, and P.~A. Parrilo}, {\em
  {SOSTOOLS}: Sum of squares optimization toolbox for {MATLAB}}, 2004,
  \url{http://www.cds.caltech.edu/sostools}.

\bibitem{RaymondSinghThomas2015}
{\sc A.~Raymond, M.~Singh, and R.~R. Thomas}, {\em Symmetry in {T}ur{\'a}n sums
  of squares polynomials from flag algebras}, arXiv preprint arXiv:1808.08431,
  (2015).

\bibitem{Rockafellar1970}
{\sc R.~T. Rockafellar}, {\em Convex Analysis}, Princeton University Press,
  Princeton, {NJ}, 1970.

\bibitem{SeidlerDeWolff2018}
{\sc H.~Seidler and T.~de~Wolff}, {\em An experimental comparison of {SONC} and
  {SOS} certificates for unconstrained optimization}, arXiv preprint
  arXiv:1808.08431,  (2018).

\bibitem{Shor1987}
{\sc N.~Z. Shor}, {\em An approach to obtaining global extremums in polynomial
  mathematical programming problems}, Cybernetics, 23 (1987), pp.~695--700.

\bibitem{Wang2019}
{\sc J.~Wang}, {\em Nonnegative polynomials and circuit polynomials}, arXiv
  preprint arXiv:1804.09455,  (2019).

\end{thebibliography}

\end{document}